\documentclass[11pt]{article}
\usepackage{amsmath,amsthm,amssymb,amsfonts,amscd}
\usepackage{amsxtra}
\usepackage{eucal}
\usepackage{mathrsfs}
\usepackage{color, xcolor}

\usepackage[a4paper, margin=2cm]{geometry}

\usepackage{hyperref, cleveref}
\usepackage{enumerate}
\usepackage{blkarray}
\usepackage{enumitem}
\usepackage{tcolorbox}
\usepackage{lmodern} 
\usepackage{caption, standalone}

\setlist[enumerate,1]{label={(\alph*)}}

\definecolor{crimson}{rgb}{0.85, 0.08, 0.23}
\definecolor{limegreen}{rgb}{0.20, 0.78, 0.20}
\definecolor{steelblue}{rgb}{0.27, 0.51, 0.70}
\hypersetup{
  colorlinks,
  citecolor=crimson,
  linkcolor=steelblue,
  urlcolor=limegreen,
  pdftitle={},
  pdfauthor={}}

\usepackage{tikz,tikz-cd}
\tikzstyle{vertex} = [shape=circle, fill=black, draw,minimum size=5pt, inner sep=0pt]
\tikzstyle{edge} = [thick]
\tikzstyle{arc} = [->,thick, > = stealth]

\usepackage{algorithm}
\usepackage{algpseudocode}

\newenvironment{subproof}[1][\proofname]{%
  \begin{proof}[#1]%
}{%
  \end{proof}%
}

\newtheorem{innercustomcase}{Case}
\newenvironment{customcase}[1]
{\renewcommand\theinnercustomcase{#1}\innercustomcase}
{\endinnercustomcase}

\newtheorem{innercustomsubcase}{Subcase}
\newenvironment{customsubcase}[1]
{\renewcommand\theinnercustomsubcase{#1}\innercustomsubcase}
{\endinnercustomsubcase}

\newtheorem{innercustomclaim}{Claim}
\newenvironment{customclaim}[1]
{\renewcommand\theinnercustomclaim{#1}\innercustomclaim}
{\endinnercustomclaim}

\newtheorem{innercustomtheorem}{Theorem}
\newenvironment{customtheorem}[1]
{\renewcommand\theinnercustomtheorem{#1}\innercustomtheorem}
{\endinnercustomtheorem}

\newtheorem{theorem}{Theorem}[section]
\newtheorem{lemma}[theorem]{Lemma}

\newtheorem{claim}[]{Claim}
\newtheorem{corollary}[theorem]{Corollary}
\newtheorem{proposition}[theorem]{Proposition}
\newtheorem{observation}[theorem]{Observation}

\theoremstyle{definition}

\theoremstyle{remark}

\usepackage{authblk}
\usepackage{lineno}

\newcommand{\Ptwo}{\Delta(1,2,2)}

\title{The structure of $\Delta(1, 2, 2)$-free tournaments}
\author[1,2]{Seokbeom Kim\thanks{Supported by the Institute for Basic Science (IBS-R029-C1).}}
\author[3]{Taite LaGrange}
\author[3]{Mathieu Rundstr\"{o}m}
\author[4]{Arpan Sadhukhan\thanks{Sadhukhan is currently affiliated with the Indian Institute of Technology, Dharwad. This work was done when he was supported by the Dutch Research Council (NWO) through Gravitation-grant NETWORKS-024.002.003.}}
\author[3]{Sophie Spirkl\thanks{We acknowledge the support of the Natural Sciences and Engineering Research Council of Canada (NSERC), [funding reference number RGPIN-2020-03912]. Cette recherche a \'et\'e financ\'ee par le Conseil de recherches en sciences naturelles et en g\'enie du Canada (CRSNG), [num\'ero de r\'ef\'erence RGPIN-2020-03912]. This project was funded in part by the Government of Ontario. This research was conducted while Spirkl was an Alfred P. Sloan Fellow.}}
\affil[1]{Department of Mathematical Sciences, KAIST, Daejeon, South~Korea}
\affil[2]{Discrete Mathematics Group, Institute for Basic Science (IBS), Daejeon,~South~Korea}
\affil[3]{University of Waterloo, Waterloo, ON, Canada}
\affil[4]{Eindhoven University of Technology, Eindhoven, Netherlands}
\affil[ ]{\small Email addresses: \texttt{seokbeom@kaist.ac.kr, tlagrang@uwaterloo.ca, 
mrundstrom@uwaterloo.ca,  ra.a.sadhukhan@iitdh.ac.in, sspirkl@uwaterloo.ca}}
\date{\today}

\begin{document}
\maketitle

\begin{abstract}
    We extend the list of tournaments $S$ for which the complete structural description for tournaments excluding $S$ as a subtournament is known.
    Specifically, let $\Delta(1, 2, 2)$ be a tournament on five vertices obtained from a cyclic triangle by substituting a two-vertex tournament for two of its vertices.
    In this paper, we show that tournaments excluding $\Delta(1, 2, 2)$ as a subtournament are either isomorphic to one of three small tournaments, obtained from a transitive tournament by reversing edges in vertex-disjoint directed paths, or obtained from a smaller tournament with the same property by applying one of two operations.
    In particular, one of these operations creates a homogeneous set that induces a subtournament isomorphic to one of three fixed tournaments, and the other creates a homogeneous pair such that their union induces a subtournament isomorphic to a fixed tournament.
    As an application of this result, we present an upper bound for the chromatic number, a lower bound for the size of a largest transitive subtournament, and a lower bound for the number of vertex-disjoint cyclic triangles for such tournaments.
    The bounds that we present are all best possible.
\end{abstract}

\section{Introduction} \label{Section:intro}
Finding an easy recipe to construct all graphs with a desired property is of major interest in structural graph theory.
In particular, there are several interesting results that deal with \emph{$H$-free} graphs for a fixed graph $H$, where a graph is $H$-free if no induced subgraph of it is isomorphic to $H$.

Simple examples on this topic are graphs forbidding a short path.
For an integer $t \geq 1$, let $P_t$ be the path with $t$ vertices.
For instance, $P_2$-free graphs can be constructed from graphs with one vertex by taking disjoint unions; $P_3$-free graphs, which are also called \emph{cluster graphs}, can be constructed from complete graphs by taking disjoint unions.
A less trivial example is the class of $P_4$-free graphs, which are also called \emph{cographs}.
It is well known that every cograph can be constructed from graphs with a single vertex by taking disjoint unions and complements (see~\cite{CLS81}, for example, for a proof).

To the best of our knowledge, there are only four connected graphs $H$, up to taking complements, other than paths such that the structure of $H$-free graphs is known; see Figure \ref{fig:1}.
Two relatively easy examples among them are the \emph{paw} and the \emph{diamond}, where the paw is a triangle with exactly one pendant vertex and the diamond is a graph on four vertices with exactly one pair of non-adjacent vertices.
Olariu~\cite{Olariu88} proved that every component of a paw-free graph is either triangle-free or a complete multipartite graph, and Wallis and Zhang~\cite{WZ90} proved that any two distinct maximal cliques in a diamond-free graph share at most one vertex.
In particular, the latter implies that a graph $G$ is diamond-free if and only if there is a bipartite graph $H$ without a cycle of length at most $6$, and with a bipartition $(A, V(G))$ such that $G$ can be obtained from $H$ by making the neighbourhood of each vertex in $A$ into a clique and deleting $A$.

The other two examples are more complicated. 
The \emph{claw} is a tree with three leaves and one non-leaf vertex and the \emph{bull} is a graph obtained from a triangle by adding two new pendant vertices at two distinct vertices of the triangle.
Claw-free graphs or bull-free graphs have more involved structure theorems, showing that they can also be constructed from graphs in some ``basic'' graph classes by applying several operations.
We refer to~\cite{CS05claw} for the structure theorem for claw-free graphs and~\cite{Chud12A,Chud12B} for bull-free graphs.

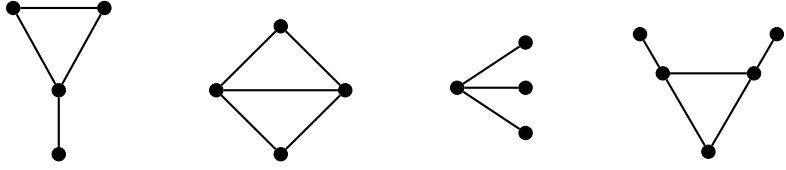
\begin{figure}[t]
    \begin{center}	
        \begin{tikzpicture}[scale=0.6]
            \tikzset{vertex/.style = {shape=circle, fill=black, draw,minimum size=5pt, inner sep=0pt}}
            \tikzset{arc/.style = {->,thick, > = stealth}}
            \tikzset{edge/.style = {thick}}
                             
            \node[vertex] (1) at (0, 0) {};
            \node[vertex] (2) at (0, 1.414) {};
            \node[vertex] (3) at (-1, 3.232) {};
            \node[vertex] (4) at (1, 3.232) {};

            \draw[edge] (1) to (2);
            \draw[edge] (2) to (3);
            \draw[edge] (2) to (4);
            \draw[edge] (3) to (4);
        \end{tikzpicture} \quad \quad \quad
        \begin{tikzpicture}[scale=0.6]
            \tikzset{vertex/.style = {shape=circle, fill=black, draw,minimum size=5pt, inner sep=0pt}}
            \tikzset{arc/.style = {->,thick, > = stealth}}
            \tikzset{edge/.style = {thick}}
                             
            \node[vertex] (1) at (0, 0) {};
            \node[vertex] (2) at (-1.414, 1.414) {};
            \node[vertex] (3) at (1.414, 1.414) {};
            \node[vertex] (4) at (0, 2.828) {};

            \draw[edge] (1) to (2);
            \draw[edge] (1) to (3);
            \draw[edge] (2) to (3);
            \draw[edge] (2) to (4);
            \draw[edge] (3) to (4);
        \end{tikzpicture} \quad \quad \quad
        \begin{tikzpicture}[scale=0.6]
            \tikzset{vertex/.style = {shape=circle, fill=black, draw,minimum size=5pt, inner sep=0pt}}
            \tikzset{arc/.style = {->,thick, > = stealth}}
            \tikzset{edge/.style = {thick}}
                             
            \node () at (0, 0) {};
            \node[vertex] (1) at (0, 0.414) {};
            \node[vertex] (2) at (-1.5, 1.414) {};
            \node[vertex] (3) at (0, 1.414) {};
            \node[vertex] (4) at (0, 2.414) {};

            \draw[edge] (1) to (2);
            \draw[edge] (2) to (3);
            \draw[edge] (2) to (4);
        \end{tikzpicture} \quad \quad \quad
        \begin{tikzpicture}[scale=0.6]
            \tikzset{vertex/.style = {shape=circle, fill=black, draw,minimum size=5pt, inner sep=0pt}}
            \tikzset{arc/.style = {->,thick, > = stealth}}
            \tikzset{edge/.style = {thick}}
            
            \node () at (0, 0) {};

            \node[vertex] (1) at (0, 0) {};
            \node[vertex] (2) at (-1, 1.732) {};
            \node[vertex] (3) at (1, 1.732) {};
            \node[vertex] (4) at (-1.5, 2.598) {};
            \node[vertex] (5) at (1.5, 2.598) {};

            \draw[edge] (1) to (2);
            \draw[edge] (2) to (3);
            \draw[edge] (1) to (3);
            \draw[edge] (2) to (4);
            \draw[edge] (3) to (5);
        \end{tikzpicture} 
    \caption{Paw, diamond, claw, and bull.} \label{fig:1}
    \end{center}
\end{figure} 

As digraphs can be thought of as a generalization of (undirected) graphs, it is natural to ask for the construction for digraphs with a desired property.
Recently, Seymour~\cite{Seymour24+} showed how to construct all directed graphs such that every directed cycle has length three, and Hatzel, Kreutzer, Protopapas, Reich, Stamoulis, and Wiederrecht~\cite{HKPRSW24+} gave a construction for all strongly $2$-connected directed graphs.
However, finding structure theorems for directed graphs is no less difficult than for graphs, so a common approach in this area has been to study the structure of a special type of directed graphs.
The goal of this paper is to find a structure theorem of directed graphs with this perspective.

\subsection{Main results}

A \emph{tournament} is a directed graph such that for every pair of two distinct vertices, there is exactly one edge between them.
Equivalently, a tournament is an orientation of a complete graph.
A tournament~$T$ \emph{contains} a tournament $S$ if there is a subtournament of $T$ isomorphic to $S$, and such a subtournament is called a \emph{copy} of $S$.
If $T$ does not contain $S$, we say $T$ is \emph{$S$-free}.
For a tournament $T$ and $X \subseteq V(T)$,~$T[X]$ denotes the subtournament of $T$ induced on $X$.

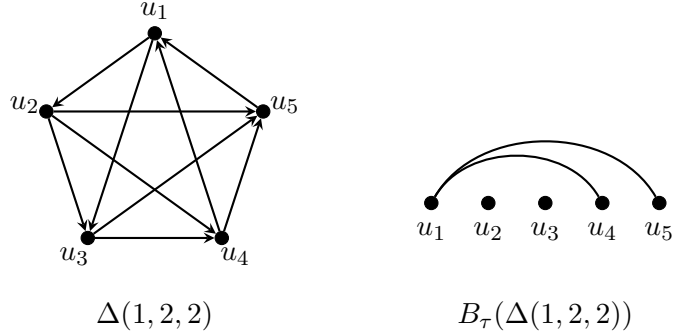
\begin{figure}[t]
    \begin{center}	
        \begin{tikzpicture}[scale=0.75]
            \tikzset{vertex/.style = {shape=circle, fill=black, draw,minimum size=5pt, inner sep=0pt}}
            \tikzset{arc/.style = {->,thick, > = stealth}}
            \tikzset{edge/.style = {thick}}
                             
            \foreach \i in {1,2,...,5}{
                  \node[vertex] (\i) at ({90+360/5*(\i-1)}:2) {};
                 }
       		\foreach \i in {1,2,...,5}{
                  \node at ({90+360/5*(\i-1)}:2.4) {$u_{\i}$};
                 }

            \node at (0, -3) {$\Delta(1, 2, 2)$};

            \draw[arc] (1) to (2);
            \draw[arc] (1) to (3);
            \draw[arc] (2) to (3);
            \draw[arc] (2) to (4);
            \draw[arc] (2) to (5);
            \draw[arc] (3) to (4);
            \draw[arc] (3) to (5);
            \draw[arc] (4) to (1);
            \draw[arc] (4) to (5);
            \draw[arc] (5) to (1);
        \end{tikzpicture} \quad \quad \quad
        \begin{tikzpicture}[scale=0.75]
            \tikzset{vertex/.style = {shape=circle, fill=black, draw,minimum size=5pt, inner sep=0pt}}
            \tikzset{arc/.style = {->,thick, > = stealth}}
            \tikzset{edge/.style = {thick}}
                             
            \foreach \i in {1,2,...,5}{
                  \node[vertex] (\i) at (\i-3, -1) {};
                 }
       		\foreach \i in {1,2,...,5}{
                  \node at (\i-3, -1.5) {$u_{\i}$};
                 }

            \node at (0, -3) {$B_\tau(\Delta(1, 2, 2))$};

            \draw[edge, bend left=60] (1) to (4);
            \draw[edge, bend left=60] (1) to (5);
        \end{tikzpicture}
    \caption{The tournament $\Delta(1, 2, 2)$ and its backedge graph with respect to the ordering $\tau = (u_1, \ldots, u_5)$.}
    \label{Fig:P2}
    \end{center}
\end{figure} 

Let $\Delta(1, 2, 2)$ be the tournament obtained from a cyclic triangle by substituting for two of its vertices a tournament with two vertices (see~\Cref{Fig:P2} for an illustration).
Our main result presents the complete structural description of $\Delta(1, 2, 2)$-free tournaments.
Before explaining our result, we first introduce the well-known concept of \emph{backedge graphs}.
As tournaments become more complicated, it is often useful to describe them in terms of their backedge graphs.

An \emph{ordering} of a graph or a tournament on $n$ vertices is an enumeration $(v_1, \ldots, v_n)$ of its vertices, and an \emph{ordered graph} is a graph with an ordering.
Let $T$ be a tournament on $n$ vertices and $\sigma = (v_1, \ldots, v_n)$ be an ordering of $T$.
We say an edge $v_i v_j \in E(T)$ is a \emph{backedge under $\sigma$} if $i > j$.
The \emph{backedge graph of $T$ with respect to $\sigma$}, denoted by $B_\sigma(T)$, is an ordered (undirected) graph with vertex set $V(T)$, edge set $\{v_i v_j : \text{$i>j$ and $v_i v_j \in E(T)$}\}$, and ordering $\sigma$.
We also say that a graph $G$ is a \emph{backedge graph} of $T$ if there is an ordering $\sigma$ of $T$ such that $G$ is isomorphic to $B_\sigma(T)$ when we forget the ordering of $B_\sigma(T)$.
For example, a tournament is transitive if and only if it admits an edgeless backedge graph. 

Given an ordered graph $G$ with ordering $(v_1, \ldots, v_n)$, a \emph{monotone path} is an induced path $v_{i(1)} \cdots v_{i(\ell)}$ in $G$, where $\ell \geq 1$ is an integer, such that the sequence $i(1), \ldots, i(\ell)$ is monotone.
For each $i \in \{5, 6, 7\}$, let $H_i$ be an ordered graph with ordering $(v_1, \ldots, v_i)$ such that
\[
E(H_5) = \{v_1 v_3, v_1v_5, v_2 v_5, v_3 v_4\}, \quad 
E(H_6) = \{v_1 v_2, v_1 v_6, v_2 v_3, v_2 v_4, v_3 v_5, v_4 v_6\},
\]
and
\[
E(H_7) = E(H_6) \cup \{v_1 v_7, v_3 v_7, v_4 v_7\}.
\]

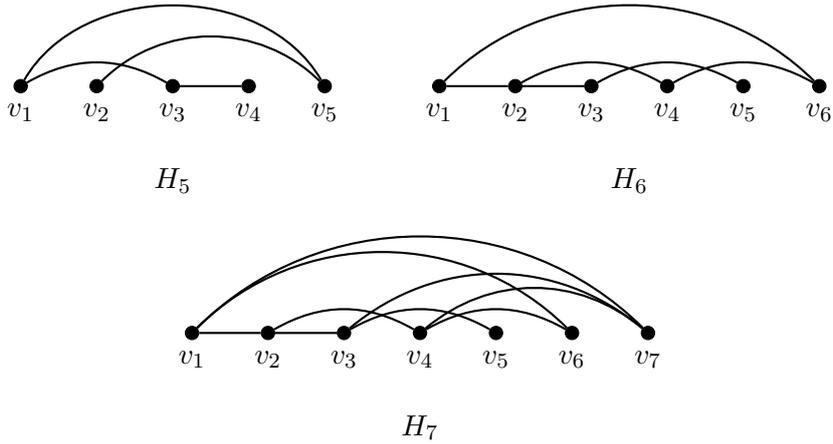
\begin{figure}[t]
    \centering
        \begin{tikzpicture}
            \tikzset{vertex/.style = {shape=circle, fill=black, draw,minimum size=5pt, inner sep=0pt}}
            \tikzset{arc/.style = {->,thick, > = stealth}}
            \tikzset{edge/.style = {thick}}
                             
            \foreach \i in {1,2,...,5}{
              \node[vertex] (\i) at (-3+\i, 0) {};
            }
             
            \foreach \i in {1,2,...,5}{
              \node at (-3+\i, -0.35) {$v_{\i}$};
            }

            \node at (0, -1.25) {$H_5$};

            \draw[edge, bend left=30] (1) to (3);
            \draw[edge, bend left=60] (1) to (5);
            \draw[edge, bend left=45] (2) to (5);
            \draw[edge] (3) to (4);
        \end{tikzpicture}  \quad\quad
        \begin{tikzpicture}
            \tikzset{vertex/.style = {shape=circle, fill=black, draw,minimum size=5pt, inner sep=0pt}}
            \tikzset{arc/.style = {->,thick, > = stealth}}
            \tikzset{edge/.style = {thick}}
                             
            \foreach \i in {1,2,...,6}{
              \node[vertex] (\i) at (-3+\i, 0) {};
            }
             
            \foreach \i in {1,2,...,6}{
              \node at (-3+\i, -0.35) {$v_{\i}$};
            }

            \node at (0.5, -1.25) {$H_6$};

            \draw[edge] (1) to (2);
            \draw[edge, bend left=45] (1) to (6);
            \draw[edge] (2) to (3);
            \draw[edge, bend left=30] (2) to (4);
            \draw[edge, bend left=30] (3) to (5);
            \draw[edge, bend left=30] (4) to (6);
        \end{tikzpicture}
        
        \begin{tikzpicture}
            \tikzset{vertex/.style = {shape=circle, fill=black, draw,minimum size=5pt, inner sep=0pt}}
            \tikzset{arc/.style = {->,thick, > = stealth}}
            \tikzset{edge/.style = {thick}}
                             
           \foreach \i in {1,2,...,7}{
              \node[vertex] (\i) at (-4+\i, 0) {};
            }
             
            \foreach \i in {1,2,...,7}{
              \node at (-4+\i, -0.35) {$v_{\i}$};
            }

            \node at (0, -1.25) {$H_7$};

            \draw[edge] (1) to (2);
            \draw[edge, bend left=45] (1) to (6);
            \draw[edge] (2) to (3);
            \draw[edge, bend left=30] (2) to (4);
            \draw[edge, bend left=30] (3) to (5);
            \draw[edge, bend left=30] (4) to (6);
            \draw[edge, bend left=45] (1) to (7);
            \draw[edge, bend left=40] (3) to (7);
            \draw[edge, bend left=40] (4) to (7);
        \end{tikzpicture} 
    \caption{Three ordered graphs $H_5$, $H_6$, and $H_7$.}
    \label{Fig:Backedge-graph-components}
\end{figure} 

See~\Cref{Fig:Backedge-graph-components} for an illustration.
Observe that all ordered graphs $H_5, H_6, H_7$ are connected; in fact, these are the backedge graphs of special tournaments, and we have chosen these (non-standard) backedge graphs because it will be convenient for stating~\Cref{Thm:main-backedge-intro} below. 
In $H_6$, the two sets of vertices $\{v_1, v_2, v_3\}, \{v_4, v_5, v_6\}$ are called the \emph{flocks} of $H_6$.

Our main theorem is as follows.
We say a component of an ordered graph $G$ is \emph{consecutive} if the vertices of the component are consecutive in the ordering of $G$.

\begin{theorem} \label{Thm:main-backedge-intro}
    Every $\Delta(1, 2, 2)$-free tournament has a backedge graph in which each component is either a monotone path or isomorphic to one of $H_5, H_6, H_7$.
    In particular, the components isomorphic to $H_5$ or $H_7$ are consecutive, and the vertices in each flock of components isomorphic to $H_6$ are consecutive in the ordering of the backedge graph.
\end{theorem}
Note that not every tournament as in Theorem \ref{Thm:main-backedge-intro} is $\Delta(1, 2, 2)$-free; for example, the ordered graph with ordering $(v_1, \dots, v_5)$ and set of edges $\{v_1v_3, v_2v_4, v_3v_5\}$ is a backedge graph of $\Delta(1,2,2).$

One might ask why we chose to use the backedge graph to describe the structure, since it may be difficult to get intuition from it. 
The reason is that this formulation represents the ``strongest" version of our main theorem, in the sense that it can be directly translated to a decomposition scheme of, or even a construction for all $\Delta(1, 2, 2)$-free tournaments.

Let us translate~\Cref{Thm:main-backedge-intro} into a decomposition scheme of $\Delta(1, 2, 2)$-free tournaments.
The crucial observation is the following: If a set $X$ of vertices induces a consecutive component in the backedge graph of a tournament $T$, then each vertex not in $X$ cannot have both an in-neighbour and an out-neighbour in $X$.
Similarly, if $(A_1, A_2)$ is a pair of flocks of a component isomorphic to~$H_6$, then each vertex not in $A_1 \cup A_2$ cannot have both an in-neighbour and an out-neighbour in each flock.
This means that $X$ and $(A_1, A_2)$ ``act like" a vertex and an edge in~$T$, respectively, and we can decompose $T$ into two subtournaments by using this information. Let us formalize this idea: Let~$T$ be a tournament. 
For two nonempty disjoint sets $X, Y \subseteq V(T)$, we write $X \Rightarrow Y$ if $xy \in E(T)$ for all $x \in X$ and $y \in Y$.
We also write $v \Rightarrow Y$ for $\{v\} \Rightarrow Y$ and $X \Rightarrow v$ for $X \Rightarrow \{v\}$.
A nonempty set~$X \subseteq V(T)$ is a \emph{homogeneous set} if each $v \in V(T) \setminus X$ satisfies $v \Rightarrow X$ or $X \Rightarrow v$. A vertex $v\in V(T)\setminus X$ is \emph{mixed} on~$X$ if neither $v\Rightarrow X$ nor $X\Rightarrow v$ holds. 
For two nonempty disjoint sets~$A_1, A_2 \subseteq V(T)$, the pair~$(A_1, A_2)$ is a \emph{homogeneous pair} if $A_1$ is a homogeneous set of $T \setminus A_2$ and~$A_2$ is a homogeneous set of~$T \setminus A_1$. 

Let $T_5$ be the tournament with vertex set $\{v_1, \ldots, v_5\}$ such that $v_i v_j \in E(T_5)$ if and only if $j-i$ is congruent to $1$ or $2$ modulo $5$.
Let $P_7$ be the Paley tournament on seven vertices, that is, the vertex set is $\{v_1, \ldots, v_7\}$ and $v_i v_j \in E(P_7)$ if and only if $j-i$ is congruent to $1$, $2$, or $4$ modulo $7$.
Let $P_7^-$ be the tournament obtained from $P_7$ by deleting a vertex.
Since $P_7$ is vertex-transitive, the chosen vertex does not matter.
We call the tournaments $T_5, P_7^-, P_7$ the \emph{basic tournaments}.
We remark that the ordered graphs $H_5, H_6, H_7$ are backedge graphs of $T_5$, $P_7^-$, and $P_7$, respectively, under appropriate orderings.

Observe that there is an ordering $\eta = (u_1, \ldots, u_6)$ of $P_7^-$ such that the backedges under this ordering are $u_1 u_3, u_1 u_6, u_2 u_5, u_3 u_4, u_4 u_6$; see Figure \ref{Figure:P_7^-}.
In particular, when an ordering $\theta$ of $P_7^-$ leads to a backedge graph isomorphic to $B_\eta(P_7^-)$, we say $\theta$ is a \emph{canonical ordering} of $P_7^-$.
The \emph{degree partition} of $P_7^-$ is the partition $(D_1, D_2)$ of $V(P_7^-)$ such that $D_1 = \{v : d^+(v)=3 \}$ and $D_2 = \{v : d^+(v) = 2\}$.
Note that if $\theta = (w_1, \ldots, w_6)$ is a canonical ordering of $P_7^-$, then $(\{w_1, w_2, w_3\}, \{w_4, w_5, w_6\})$ is the degree partition of the tournament.

Now, we can translate~\Cref{Thm:main-backedge-intro} as follows.

\begin{figure}[t]
    \begin{center}	
    \captionsetup{justification=centering}
        \begin{tikzpicture}[scale=0.75]
            \tikzset{vertex/.style = {shape=circle, fill=black, draw,minimum size=5pt, inner sep=0pt}}
            \tikzset{arc/.style = {->,thick, > = stealth}}
            \tikzset{edge/.style = {thick}}
                             
            \foreach \i in {1,2,...,6}{
                  \node[vertex] (\i) at ({90+360/6*(\i-1)}:2) {};
                 }
       		\foreach \i in {1,2,...,6}{
                  \node at ({90+360/6*(\i-1)}:2.5) {$u_{\i}$};
                 }

            \node at (0, -3) {$P_7^-$};
            
            \draw[arc] (3) to (1);
            \draw[arc] (6) to (1);
            \draw[arc] (5) to (2);
            \draw[arc] (4) to (3);
            \draw[arc] (6) to (4);

            \draw[arc] (1) to (2);
            \draw[arc] (1) to (4);
            \draw[arc] (1) to (5);
            \draw[arc] (2) to (3);
            \draw[arc] (2) to (4);
            \draw[arc] (2) to (6);
            \draw[arc] (3) to (5);
            \draw[arc] (3) to (6);
            \draw[arc] (4) to (5);
            \draw[arc] (5) to (6);
        \end{tikzpicture} \quad \quad \quad
        \begin{tikzpicture}[scale=0.75]
            \tikzset{vertex/.style = {shape=circle, fill=black, draw,minimum size=5pt, inner sep=0pt}}
            \tikzset{arc/.style = {->,thick, > = stealth}}
            \tikzset{edge/.style = {thick}}
                             
            \foreach \i in {1,2,...,6}{
                  \node[vertex] (\i) at (\i-3, -1) {};
                 }
       		\foreach \i in {1,2,...,6}{
                  \node at (\i-3, -1.5) {$u_{\i}$};
                 }

            \node at (0, -3) {$B_\eta(P_7^-)$};

            \draw[edge, bend left=60] (1) to (3);
            \draw[edge, bend left=60] (1) to (6);
            \draw[edge, bend left=60] (2) to (5);
            \draw[edge, bend left=60] (3) to (4);
            \draw[edge, bend left=60] (4) to (6);
        \end{tikzpicture}
    \caption{Tournament $P_7^-$ and its backedge graph with respect to the ordering $\eta = (u_1, \ldots, u_6)$. $(\{u_1, u_2, u_3\}, \{u_4, u_5, u_6\})$ is the degree partition of $P_7^-$.}
    \label{Figure:P_7^-}
    \end{center}
\end{figure}
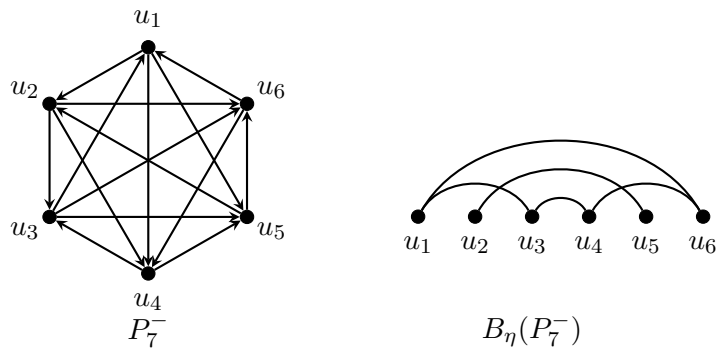 

\begin{theorem} \label{Thm:main-intro}
    Let $T$ be a $\Delta(1, 2, 2)$-free tournament.
    Then one of the following holds:
    \begin{enumerate}
        \item $T$ is isomorphic to a basic tournament;
        \item $T$ is obtained from a transitive tournament by reversing edges in vertex-disjoint directed paths;
        \item $T$ contains a homogeneous set $X$ such that $T[X]$ is isomorphic to $T_5$ or $P_7$; or \label{item:homogeneous-set}
        \item $T$ contains a homogeneous pair $(D_1, D_2)$ such that $G[D_1 \cup D_2]$ is isomorphic to $P_7^-$ and $(D_1, D_2)$ is the degree partition of this subtournament. \label{item:homogeneous-pair}
    \end{enumerate}
\end{theorem}

We emphasize that~\Cref{Thm:main-intro} also gives a construction for all $\Delta(1, 2, 2)$-free tournaments.
However, we postpone stating the formulation that describes the construction more explicitly until \Cref{Subsection:main} since the operations suggested by~\labelcref{item:homogeneous-set} and~\labelcref{item:homogeneous-pair} in~\Cref{Thm:main-intro} do not always preserve being~$\Delta(1, 2, 2)$-free.

As an application of~\Cref{Thm:main-intro}, we determine tight bounds of several tournament parameters for~$\Delta(1, 2, 2)$-free tournaments.
For an integer $k \geq 1$, a \emph{$k$-colouring} of a tournament $T$ is a map~$f : V(T) \to [k]$ such that $\{v : f(v) = i\}$ induces a transitive subtournament for each $i \in [k]$.
A tournament~$T$ is \emph{$k$-colourable} if it admits a $k$-colouring, and the \emph{chromatic number} of $T$, denoted by~$\chi(T)$, is the smallest integer $k$ such that $T$ is $k$-colourable.

The chromatic number of tournaments has been widely studied due to its connection with the celebrated Erd\H{o}s-Hajnal conjecture~\cite{EH89} and its equivalent formulation in terms of tournaments~\cite{APS01}.
In their seminal paper~\cite{BCCFLSST13}, Berger, Chudnovsky, Choromanski, Fox, Loebl, Scott, Seymour, and Thomass\'{e} characterized \emph{heroes}, which are tournaments $S$ such that $S$-free tournaments have bounded chromatic number.
Their result implies that $\Delta(1, 2, 2)$ is a hero, but the exact upper bound for the chromatic number of $\Delta(1, 2, 2)$-free tournaments was not previously known.

\begin{theorem} \label{Thm:colouring-intro}
    Every $\Delta(1, 2, 2)$-free tournament is $3$-colourable, and it is $2$-colourable if and only if it is $P_7$-free.
\end{theorem}

The other application deals with the order of a largest transitive subtournament of a tournament $T$, denoted by $\vec{\alpha}(T)$.
\Cref{Thm:colouring-intro} shows that every $\Delta(1, 2, 2)$-free tournament $T$ satisfies $\vec{\alpha}(T) \geq \lvert V(T) \rvert/3$ as it is $3$-colourable.
However, this bound is far from being optimal, and we prove the optimal lower bound by combining~\Cref{Thm:main-backedge-intro} and~\Cref{Thm:colouring-intro}.

\begin{theorem} \label{Thm:transitive-intro}
    Every $\Delta(1, 2, 2)$-free tournament $T$ with $n$ vertices satisfies $\vec{\alpha}(T) \geq \frac{3n}{7}$.
\end{theorem}

The last application is related to the \emph{triangle packing}.
A triangle packing in a tournament $T$ is a set of vertex-disjoint cyclic triangles in $T$, and we denote by $\nu(T)$ the maximum size of a triangle packing in $T$.
This parameter naturally arises in the sense of duality, since it is a dual parameter of the minimum size of the feedback vertex set in tournaments~\cite{BBT17}.

\begin{theorem}[Informal; see~\Cref{Thm:triangle-packing} for the formal statement] \label{Thm:triangle-packing-intro}
    Let $T$ be a $\Ptwo$-free tournament and let $\sigma$ be an ordering of $T$ ``that witnesses the construction of $T$."
    If $B_\sigma(T)$ has $m$ edges, then $\nu(T)\ge\frac{2m}{7}$. 
\end{theorem}

We give the proof of~\Cref{Thm:colouring-intro},~\Cref{Thm:transitive-intro}, and~\Cref{Thm:triangle-packing-intro} in Section~\ref{Section:Application}.

\subsection{Previous results on \texorpdfstring{$S$}{S}-free tournaments}

When a tournament $S$ has at most three vertices, the structure of $S$-free tournaments is well-understood.
For example, if $S$ is a cyclic triangle, then $S$-free tournaments are exactly the transitive tournaments; if $S$ is the transitive tournament with three vertices, then an $S$-free tournament either has at most two vertices or is isomorphic to a cyclic triangle.
However, even when $S$ has four vertices, the structure of $S$-free tournaments becomes nontrivial.
The simplest example would be the (unique) strongly connected tournament with four vertices, say $SC_4$.
Since every strongly connected tournament with at least four vertices contains $SC_4$ by Moon's Theorem~\cite{Moon71}, any strongly connected component of $SC_4$-free tournaments has at most three vertices.

However, except for these simple examples, there are only three tournaments $S$, up to reversing all edges, such that the complete structure of $S$-free tournaments is known.
To state these results, we need to define some terminology.
A homogeneous set $X$ of a tournament $T$ is \emph{trivial} if $\lvert X \rvert = 1$ or~$X = V(T)$, and \emph{nontrivial} otherwise.
A tournament is \emph{prime} if it has no nontrivial homogeneous sets.

Why do we care about prime tournaments?
If a tournament $T$ has a nontrivial homogeneous set, then we can obtain $T$ from two smaller tournaments, both of which are subtournaments of $T$, by the substitution operation (see \Cref{Section:Substitutions} for the definition of substitution).
In particular, when $S$ is a prime tournament, then $T$ is $S$-free if and only if both of these subtournaments are $S$-free.
Thus, if we know the complete structure of prime $S$-free tournaments, we can construct all $S$-free tournaments by substitution.

For each odd integer $n \geq 1$, define three tournaments $T_n$, $U_n$, and $W_n$ as follows  (see~\Cref{Fig:Simple-prime-tournaments} for an illustration).
Let $k = (n-1)/2$.
\begin{itemize}
    \item $T_n$ is a tournament with vertices $v_1, \ldots, v_n$ and $v_i v_j \in E(T_n)$ if and only if $j-i$ is congruent to $1, \ldots, k$ modulo $n$;
    \item $U_n$ is obtained from $T_n$ by reversing all edges with both ends in $\{v_2, \ldots, v_{k+1}\}$; and
    \item $W_n$ is a tournament with vertices $v_1, \ldots, v_{n-1}, w$ such that $v_i v_j \in E(W_n)$ for each $1 \leq  i < j \leq n-1$ and $\{v_i : \text{$i$ is even}\} \Rightarrow w \Rightarrow \{v_i : \text{$v_i$ is odd} \}$.
\end{itemize}
In particular, each tournament $T_n, U_n, W_n$ is a one-vertex tournament when $n=1$, a cyclic triangle when $n=3$, and a prime tournament for every odd integer $n \geq 1$.

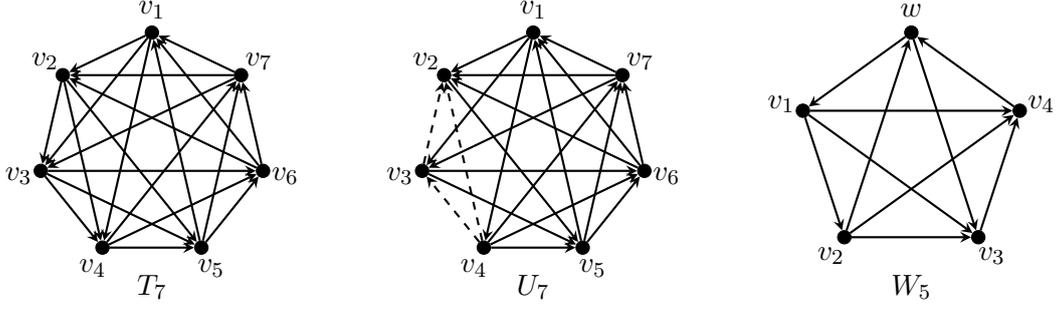
\begin{figure}[t]
    \centering
        \begin{tikzpicture}[scale=0.75]
            \tikzset{vertex/.style = {shape=circle, fill=black, draw,minimum size=5pt, inner sep=0pt}}
            \tikzset{arc/.style = {->,thick, > = stealth}}
            \tikzset{edge/.style = {thick}}
                             
               \foreach \i in {1,2,...,7}{
              \node[vertex] (\i) at ({90+360/7*(\i-1)}:2) {};
             }
             
               \foreach \i in {1,2,...,7}{
              \node at ({90+360/7*(\i-1)}:2.4) {$v_{\i}$};
             }

            \node at (0, -2.5) {$T_7$};
             
            \draw[arc] (1) to (2);
            \draw[arc] (1) to (3);
            \draw[arc] (1) to (4);
            \draw[arc] (2) to (3);
            \draw[arc] (2) to (4);
            \draw[arc] (2) to (5);
            \draw[arc] (3) to (4);
            \draw[arc] (3) to (5);
            \draw[arc] (3) to (6);
            \draw[arc] (4) to (5);
            \draw[arc] (4) to (6);
            \draw[arc] (4) to (7);
            \draw[arc] (5) to (6);
            \draw[arc] (5) to (7);
            \draw[arc] (5) to (1);
            \draw[arc] (6) to (7);
            \draw[arc] (6) to (1);
            \draw[arc] (6) to (2);
            \draw[arc] (7) to (1);
            \draw[arc] (7) to (2);
            \draw[arc] (7) to (3);
        \end{tikzpicture}  \quad  \quad           
        \begin{tikzpicture}[scale=0.75]
            \tikzset{vertex/.style = {shape=circle, fill=black, draw,minimum size=5pt, inner sep=0pt}}
            \tikzset{arc/.style = {->,thick, > = stealth}}
            \tikzset{edge/.style = {thick}}
                             
               \foreach \i in {1,2,...,7}{
              \node[vertex] (\i) at ({90+360/7*(\i-1)}:2) {};
             }
             
               \foreach \i in {1,2,...,7}{
              \node at ({90+360/7*(\i-1)}:2.4) {$v_{\i}$};
             }
             
            \node at (0, -2.5) {$U_7$};
             
            \draw[arc] (1) to (2);
            \draw[arc] (1) to (3);
            \draw[arc] (1) to (4);
            \draw[arc, dashed] (3) to (2);
            \draw[arc, dashed] (4) to (2);
            \draw[arc] (2) to (5);
            \draw[arc, dashed] (4) to (3);
            \draw[arc] (3) to (5);
            \draw[arc] (3) to (6);
            \draw[arc] (4) to (5);
            \draw[arc] (4) to (6);
            \draw[arc] (4) to (7);
            \draw[arc] (5) to (6);
            \draw[arc] (5) to (7);
            \draw[arc] (5) to (1);
            \draw[arc] (6) to (7);
            \draw[arc] (6) to (1);
            \draw[arc] (6) to (2);
            \draw[arc] (7) to (1);
            \draw[arc] (7) to (2);
            \draw[arc] (7) to (3);
        \end{tikzpicture} \quad \quad            
        \begin{tikzpicture}[scale=0.75]
            \tikzset{vertex/.style = {shape=circle, fill=black, draw,minimum size=5pt, inner sep=0pt}}
            \tikzset{arc/.style = {->,thick, > = stealth}}
            \tikzset{edge/.style = {thick}}
                             
               \foreach \i in {1,2,...,5}{
              \node[vertex] (\i) at ({90+360/5*(\i-1)}:2) {};
             }

            \node at ({90+360/5*(2-1)}:2.4) {$v_{1}$};
            \node at ({90+360/5*(3-1)}:2.4) {$v_{2}$};
            \node at ({90+360/5*(4-1)}:2.4) {$v_{3}$};
            \node at ({90+360/5*(5-1)}:2.4) {$v_{4}$};
            \node at ({90+360/5*(1-1)}:2.4) {$w$};
             
            \node at (0, -2.5) {$W_5$};

            \draw[arc] (2) to (3);
            \draw[arc] (2) to (4);
            \draw[arc] (2) to (5);
            \draw[arc] (3) to (4);
            \draw[arc] (3) to (5);
            \draw[arc] (4) to (5);
            \draw[arc] (5) to (1);
            \draw[arc] (3) to (1);
            \draw[arc] (1) to (2);
            \draw[arc] (1) to (4);
        \end{tikzpicture} 
    \caption{Tournaments $T_7$, $U_7$, and $W_5$. Dashed edges indicated the ones reversed to obtain $U_7$ from $T_7$.}
    \label{Fig:Simple-prime-tournaments}
\end{figure} 

The following describes the structure of $W_5$-free tournaments.

\begin{theorem}[Latka~\cite{Latka03}]
    A prime tournament with at least three vertices is $W_5$-free if and only it is isomorphic to one of $P_7^-$, $P_7$, $T_n$, or $U_n$ for some odd integer $n \geq 3$.
\end{theorem}

For two tournaments $S_1$ and $S_2$, let $S_1 \Rightarrow S_2$ be the tournament obtained from the disjoint union of $S_1$ and $S_2$ by adding edges from every vertex in $S_1$ to every vertex in $S_2$.
The next result forbids the tournament $T_1 \Rightarrow T_3$.
Observe that $T_1 \Rightarrow T_3$ is not a prime tournament, so describing only the structure of prime $(T_1 \Rightarrow T_3)$-free tournaments does not directly imply the general structure theorem.
However, since the $(T_1 \Rightarrow T_3)$-free condition implies that the out-neighbourhood of each vertex induces a transitive subtournament, one can deduce the structure theorem for $(T_1 \Rightarrow T_3)$-free tournament from the information about prime $(T_1 \Rightarrow T_3)$-free tournaments.

\begin{theorem}[Liu~\cite{Liu12}]
    A prime tournament with at least three vertices is $(T_1 \Rightarrow T_3)$-free if and only if it is isomorphic to $T_n$ for some odd integer $n \geq 3$.
    Moreover, a tournament (which is not necessarily prime) is $(T_1 \Rightarrow T_3)$-free if and only if it can be obtained from $T_n$ or $T_n \Rightarrow T_1$, where $n \geq 1$ is odd, by substituting a transitive tournament for each vertex.
\end{theorem}

The last result is about $U_5$-free tournaments.

\begin{theorem}[Liu~\cite{Liu15}] \label{Thm:U5-free}
    A prime tournament with at least three vertices is $U_5$-free if and only if either it is isomorphic to $T_n$ for some odd integer $n \geq 3$ or there is a partition $(X, Y, Z)$ of its vertex set such that each $X \cup Y, Y \cup Z, Z \cup X$ induces a transitive subtournament.
\end{theorem}

\subsection*{Organization}
The rest of the paper is organized as follows.
In \Cref{Section:Prelim}, we give some definitions and another description of~\Cref{Thm:main-intro} from a constructive viewpoint.
In \Cref{Section:homogeneous-sets}, we investigate the properties of a vertex subset of a $\Delta(1, 2, 2)$-free tournament which induces a copy of $T_5$, $P_7^-$, or $P_7$.
In \Cref{Section:Finding-Paving-Orderings}, we prove a ``reshuffling" lemma useful when finding an ordering of tournament with nice properties and deduce some simple cases of~\Cref{Thm:main-intro}.
In \Cref{Section:main-theorem}, we complete the proof of the structure theorem for $\Delta(1, 2, 2)$-free tournaments.
In \Cref{Section:backedge-structure}, we translate the strctural information of $\Delta(1, 2, 2)$-free tournaments in the form of backedge graphs.
In \Cref{Section:Application}, we prove~\Cref{Thm:colouring-intro} and~\Cref{Thm:transitive-intro}.

\section{Preliminaries} \label{Section:Prelim}

\subsection{Graphs, digraphs, and ordered graphs}
All graphs and directed graphs in this paper are finite and simple.
For a positive integer $n$, let~$[n] = \{1, \ldots, n\}$.
Let $T$ be a tournament and $X \subseteq V(T)$.
We write $T \setminus X$ for $T[V(T) \setminus X]$ and $T \setminus v$ for~$T \setminus \{v\}$.
For two distinct vertices $u, v \in V(T)$, we write $uv \in E(T)$ if there is an edge from $u$ to~$v$ in $T$ and say $u$ is \emph{adjacent to $v$} and $v$ is \emph{adjacent from $u$}.
We also say that $v$ is an \emph{out-neighbour} of $u$ and $u$ is an \emph{in-neighbour} of $v$.
For $v \in V(T)$, the \emph{out-neighbourhood} and \emph{in-neighbourhood} of~$v$, denoted by $N_T^+(v)$ and $N_T^-(v)$, are the sets of out-neighbours and in-neighbours of $v$, respectively.
The \emph{out-degree} (respectively, \emph{in-degree}) of $v$, denoted by $d_T^+(v)$ (respectively, $d_T^-(v)$), is the number of out-neighbours (respectively, in-neighbours) of $v$.
For $A \subseteq V(T)$, we write $N_A^+(v) = A \cap N_T^+(v)$ and $N_A^-(v) = A \cap N_T^-(v)$.
We say $v$ is \emph{out-complete to $A$} if $A \subseteq N_T^+(v)$ and \emph{in-complete from $A$} if~$A \subseteq N_T^-(v)$.
We omit the subscript $T$ or $A$ when it is clear from the context.

Let $G$ and $H$ be ordered graphs.
The \emph{underlying graph} of~$G$ is an (unordered) graph obtained from~$G$ by forgetting the order of~$G$.
An ordered graph is \emph{connected} if its underlying graph is connected.
We say $H$ is an \emph{induced subgraph} of $G$ if the underlying graph of $H$ is an induced subgraph of the underlying graph of $G$ and the ordering of $H$ is inherited from the ordering of $G$.
Two ordered graphs~$G$ and~$H$ are \emph{isomorphic} if there is a bijection $f: V(G) \to V(H)$ that preserves the adjacency and orderings, that is,~$uv \in E(G)$ if and only if~$f(u) f(v) \in E(H)$ and~$u$ precedes~$v$ in the ordering of~$G$ if and only if~$f(u)$ precedes~$f(v)$ in the ordering of~$H$.
A \emph{copy} of~$H$ in~$G$ is an induced subgraph of~$G$ isomorphic to~$H$.
A \emph{component} of~$G$ is a maximal connected induced subgraph of~$G$.

\subsection{Substitutions and join operations} \label{Section:Substitutions}

\emph{Substitution} is a classical operation that builds larger tournaments from smaller ones.
Let $S_1$ and $S_2$ be tournaments and $v \in V(S_1)$.
We say a tournament $T$ is obtained from $S_1$ by \emph{substituting $S_2$ for $v$} if the following holds:

\begin{itemize}
	\item $V(T)$ is the disjoint union of $V(S_1) \setminus \{v\}$ and $V(S_2)$;
	\item $T[V(S_1) \setminus \{v\}]$ is isomorphic to $S_1 \setminus v$;
	\item $T[V(S_2)]$ is isomorphic to $S_2$; and
	\item for each $x_1 \in V(S_1) \setminus \{v\}$ and $x_2 \in V(S_2)$, $x_1 x_2 \in E(T)$ if and only if $x_1 v \in E(S_1)$.
\end{itemize}

Observe that if~$T$ is obtained from~$S_1$ by substituting~$S_2$ for~$v \in V(S_1)$, then~$V(S_2)$ becomes a homogeneous set of~$T$.
Indeed, the converse is also true.
For a tournament $T$ and a homogeneous set~$X \subseteq V(T)$, let~$T/X$ be the tournament isomorphic to~$T[(V(T) \setminus X) \cup \{x\}]$ for some~$x \in X$.
Since~$X$ is a homogeneous set, this notion is well-defined and does not depend on the choice of~$x$.
Then it can be shown that if $T$ has a homogeneous set~$X$, then~$T$ can be obtained from~$T/X$ by substituting~$T[X]$ for~$x$.

Now, we define an operation creating homogeneous pairs.
Let~$S$ be a tournament,~$(V_1, V_2)$ be a partition of~$V(S)$,~$S_1 = S[V_1]$, and~$S_2 = S[V_2]$.
We say a tournament~$T$ is obtained from a tournament~$J$ by \emph{joining~$S$ to an edge~$uv \in E(J)$ along~$(V_1, V_2)$} if the following hold:

\begin{itemize}
    \item $V(T)$ is the disjoint union of $V(J) \setminus \{u, v\}$ and $V(S)$;
    \item $T \setminus V_1$ is isomorphic to the tournament obtained from $J \setminus u$ by substituting $S_2$ for $v$;
    \item $T \setminus V_2$ is isomorphic to the tournament obtained from $J \setminus v$ by substituting $S_1$ for $u$; and
    \item $S = T[V(S)]$. 
\end{itemize}

Observe that the join operation creates a homogeneous pair, and if a tournament contains a homogeneous pair, then it can be obtained from two smaller tournaments by applying the join operation.

In particular, when~$S = P_7^-$ and~$(V_2, V_1)$ is the degree partition of~$P_7^-$, we call the operation of joining $S$ to an edge~$uv$ along~$(V_1, V_2)$ the~\emph{$P_7^-$-join}.
Note that we changed the order of the degree partition when defining the~$P_7^-$-join.
This is because, at the end, we will replace some \emph{backedges} with a copy of~$H_6$ in the backedge graph by applying a~$P_7^-$-join.

\subsection{Constructing \texorpdfstring{$\Delta(1, 2, 2)$}{(1, 2, 2)}-free tournaments from paving tournaments} \label{Subsection:main}

To state our main theorem more precisely, we define a special type of tournaments.
A tournament~$T$ with~$n$ vertices is called a \emph{paving tournament} if there exists an ordering~$\sigma = (v_1, \ldots, v_n)$, called a \emph{paving ordering}, of~$T$ such that~$B_\sigma(T)$ satisfies the following conditions:
\begin{enumerate}[label = (P\arabic*)]
	\item for each $i \in [n-1]$, the two vertices $v_i$ and $v_{i+1}$ are not adjacent; and \label{Condition:Paving1}
	\item for each $i \in [n]$, the vertex $v_i$ has at most one neighbour in $\{v_s : s < i\}$ and at most one neighbour in $\{v_t : t > i\}$. \label{Condition:Paving2}
\end{enumerate}
Observe that every paving tournament is obtained from a transitive tournament by reversing edges in vertex-disjoint directed paths.

A vertex $v$ of a tournament $T$ is \emph{nice} if there are no three distinct vertices $x, y_1, y_2$ in $T$ such that both $v, x, y_1$ and $v, x, y_2$ induce cyclic triangles.
An edge $uv$ of a tournament $T$ with $n$ vertices is called a \emph{bridge} if there is an ordering $\sigma = (v_1, \ldots, v_n)$ of $T$ such that $u = v_i$ and $v = v_j$, and the following hold:

\begin{itemize}
    \item $i > j$; 
	\item $uv$ is an isolated edge in $B_\sigma(T)$ (so $uv$ is a backedge under $\sigma$ and $u, v$ each have degree 1 in $B_{\sigma}(T)$);
	\item the edges between $\{v_s : s < i\}$ and $\{v_t : t > i\}$ in $B_\sigma(T)$ form  a matching; and
	\item the edges between $\{v_s : s < j\}$ and $\{v_t : t > j\}$ in $B_\sigma(T)$ form a matching. 
\end{itemize}
Observe that if $T$ is a paving tournament and $\sigma$ is a paving ordering of $T$, then every isolated edge in~$B_\sigma(T)$ is a bridge in $T$.

Finally, we are ready to state our main theorem as an explicit construction.

\begin{theorem} \label{Thm:main}
	Let $T$ be a $\Delta(1, 2, 2)$-free tournament.
	Then one of the following holds:
	\begin{enumerate}
		\item $T$ is isomorphic to a basic tournament;
		\item $T$ is a $\Delta(1,2,2)$-free paving tournament;
        \item $T$ is obtained from a smaller $\Delta(1, 2, 2)$-free tournament by substituting a basic tournament for a nice vertex; or \label{itemc}
		\item $T$ is obtained from a smaller $\Delta(1, 2, 2)$-free tournament by applying a $P_7^-$-join to a bridge. \label{itemd}
	\end{enumerate}
\end{theorem}

One caveat remains: The tournament $\Delta(1, 2, 2)$ admits exactly one paving ordering, say $(v_1, \dots, v_5)$ where $v_3v_1, v_5v_3, v_4v_2$ are backedges. 
Therefore, we need to say ``$\Delta(1,2,2)$-free paving tournament'' in the result above. 
With this restriction, every tournament constructed as in~\Cref{Thm:main} is $\Delta(1,2,2)$-free; we will show that the operations in~\ref{itemc} and~\ref{itemd} preserve being $\Delta(1,2,2)$-free.  

\section{Homogeneous structures in \texorpdfstring{$\Delta(1, 2, 2)$}{(1, 2, 2)}-free tournaments} \label{Section:homogeneous-sets}

In this section, we figure out the properties of a set of vertices in a $\Delta(1, 2, 2)$-free tournament that induces a copy of a basic tournament.

\begin{figure}[t]
    \begin{center}	
        \begin{tikzpicture}[scale=0.75]
            \tikzset{vertex/.style = {shape=circle, fill=black, draw,minimum size=5pt, inner sep=0pt}}
            \tikzset{arc/.style = {->,thick, > = stealth}}
            \tikzset{edge/.style = {thick}}
                             
            \foreach \i in {1,2,...,5}{
                  \node[vertex] (\i) at ({90+360/5*(\i-1)}:2) {};
                 }
       		\foreach \i in {1,2,...,5}{
                  \node at ({90+360/5*(\i-1)}:2.4) {$v_{\i}$};
                 }

            \draw[arc] (1) to (2);
            \draw[arc] (1) to (3);
            \draw[arc] (2) to (3);
            \draw[arc] (2) to (4);
            \draw[arc] (3) to (4);
            \draw[arc] (3) to (5);
            \draw[arc] (4) to (5);
            \draw[arc] (4) to (1);
            \draw[arc] (5) to (1);
            \draw[arc] (5) to (2);
        \end{tikzpicture} \quad\quad
        \begin{tikzpicture}[scale=0.75]
            \tikzset{vertex/.style = {shape=circle, fill=black, draw,minimum size=5pt, inner sep=0pt}}
            \tikzset{arc/.style = {->,thick, > = stealth}}
            \tikzset{edge/.style = {thick}}
                             
            \foreach \i in {1,2,...,5}{
                  \node[vertex] (\i) at (-4+2*\i, 0) {};
                 }
       		\foreach \i in {1,2,...,5}{
                  \node at (-4+2*\i, -0.5) {$v_{\i}$};
                 }

            \draw[edge, bend left=30] (1) to (4);
            \draw[edge, bend left=30] (2) to (5);
            \draw[edge, bend left=45] (1) to (5);
        \end{tikzpicture} 
    \caption{The tournament $T_5$ and the backedge of $T_5$ with respect to the ordering $(v_1, \ldots, v_5)$}
    \label{Figure:T5}
    \end{center}
\end{figure}
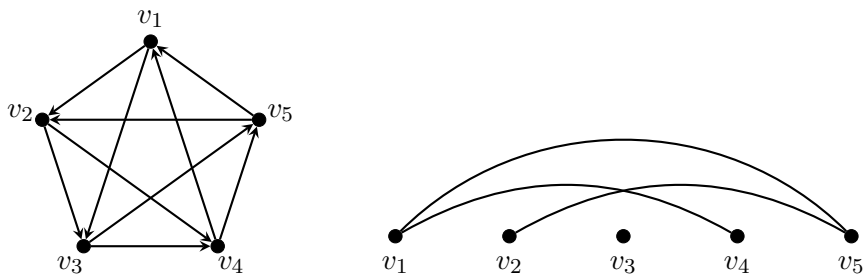 

\begin{lemma} \label{Lemma:homogeneous-C5}
    Let $T$ be a $\Ptwo$-free tournament. Suppose there exists a subset $X \subseteq V(T)$ such that $T[X]$ is isomorphic to $T_5$. Then $X$ is a homogeneous set of $T$.
\end{lemma}
\begin{proof} 
Let $X \subseteq V(T)$ be a set of vertices that induces a copy of $T_5$.
Let $X = \{v_1, \ldots, v_5\}$ such that the edges in $T[X]$ which are backedges under the ordering $(v_1, \ldots, v_5)$ are $v_4 v_1$, $v_5 v_1$, and $v_5 v_2$. 
See~\Cref{Figure:T5} for an illustration.
For convenience, we refer to the edges between $v_i$ and $v_j$ for $j>i$ as \emph{short} if $j-i \equiv 1 \pmod{5}$, and to the other type of edge as \emph{long}. Note that the automorphism group of $T_5$ acts transitively on the set of short edges as well as the set of long edges. 

Suppose for a contradiction that there is a vertex $u \in V(T) \setminus X$ such that $u$ has both an in-neighbour and an out-neighbour in $X$. 
Since $T_5$ and $\Ptwo$ are invariant under reversing all edges, it suffices to consider the cases where $u$ has at most two in-neighbours in $X$.
Without loss of generality, assume that $v_1$ is an in-neighbour of $v$ in $X$.
If $v_1$ is the unique in-neighbour of $u$ in $X$, then~$v_1 \Rightarrow \{u, v_3\} \Rightarrow \{v_4, v_5\} \Rightarrow v_1$ in $T$, so $T$ contains a copy of $\Ptwo$.
Thus, $u$ has two in-neighbours in $X$.
If those in-neighbours are the ends of a short edge, we may assume that $N_X^-(u) = \{v_1, v_2\}$.
However, then~$v_5\Rightarrow \{v_1, v_2\} \Rightarrow \{v_3, u\} \Rightarrow v_5$ in $T$, so $T$ contains a copy of $\Ptwo$.
Similarly, if those in-neighbours the ends of a long edge, then we may assume that $N_X^-(u) = \{v_1, v_3\}$, we have~$v_1 \Rightarrow \{v_3, u\} \Rightarrow \{v_4, v_5\}$ in~$T$.
Thus, $T$ again contains a copy of $\Delta(1, 2, 2)$, a contradiction.
Therefore, the set $X$ is a homogeneous set in $T$.
\end{proof}

\begin{lemma} \label{Lemma:no-mixed-on-C6}
    Let $T$ be a $\Ptwo$-free tournament. 
    Suppose that $T$ contains a copy of $P_7^-$, and let~$(A_1, A_2)$ be the degree partition of the copy.
    Then for any $v \in V(T) \setminus (A_1 \cup A_2)$, the vertex $v$ is not mixed on $A_1$ nor is $v$ mixed on $A_2$.
\end{lemma}

\begin{proof}
    Suppose the conclusion does not hold.
    By symmetry between $A_1$ and $A_2$ after reversing edges, and the fact that $\Ptwo$ is invariant under reversing edges, we may assume there is a vertex $v \in V(T)\backslash (A_1\cup A_2)$ mixed on $A_1$. 
    Label the vertices in the copy of $P_7^-$ by $a_1, a_2, \ldots, a_6$ so that~$(a_1, \ldots, a_6)$ is the canonical ordering of the copy.
    Then $A_1 = \{a_1, a_2, a_3\}$ and $A_2 = \{a_4, a_5, a_6\}$.
    
    Suppose that $v$ has two in-neighbours in $A_1$. 
    Without loss of generality, assume $a_1$ and $a_3$ are in-neighbours of $v$. 
    Since $v$ is mixed on $A_1$, $a_2$ is an out-neighbour of $v$. 
    If $a_4$ is out-neighbour of $v$, then~$a_3\Rightarrow \{a_1,v\}\Rightarrow\{a_2,a_4\}\Rightarrow a_3$; 
    otherwise, we have $a_2\Rightarrow \{a_3, a_4\}\Rightarrow\{a_5,v\}\Rightarrow a_2$.
    Thus, $v$ has exactly one in-neighbour in $A_1$, say $a_1$ without losing the generality.
    However, if $a_6$ is an out-neighbour of $v$, then $a_1\Rightarrow \{a_2,v\}\Rightarrow\{a_3,a_6\}\Rightarrow a_1$; otherwise, we have $a_3\Rightarrow \{a_1, a_6\}\Rightarrow\{a_4,v\}\Rightarrow a_3$.
    Therefore, we find a copy of $\Delta(1, 2, 2)$ in each case, a contradiction.
\end{proof}

Since deleting any vertex in $P_7$ results in a tournament isomorphic to $P_7^-$, we obtain the following.

\begin{corollary} \label{Corollary:homogeneous-C7}
    Let $T$ be a $\Ptwo$-free tournament. 
    Suppose there exists a subset $X \subseteq V(T)$ such that $T[X]$ is isomorphic to $P_7$. 
    Then $X$ is a homogeneous set of $T$.
\end{corollary}

\begin{proof}
    Suppose to the contrary that there is a vertex $v \in V(T)$ which is mixed on $X$.
    Fix an out-neighbour $x$ of $v$ in $X$ and let $A = X \setminus \{x\}$.
    Since $A$ induces a copy of $P_7^-$ in $T$, by~\Cref{Lemma:no-mixed-on-C6}, it follows that $v$ has zero, three, or six out-neighbours in $A$. 
    However, if $v$ has six out-neighbours in $A$, then $x$ is not mixed on $X$. 
    Thus, we may assume that $v$ has zero or three out-neighbours in $A$. 
    Take an in-neighbour $y$ of $v$ in $X$ and let $B = X \setminus \{y\}$.
    Observe that $N_B^+(v) = N_A^+(v) \cup \{x\}$, so $v$ has one or four out-neighbours in $B$.
    However, since $B$ induces a copy of $P_7^-$ in $T$, we reach a contradiction by~\Cref{Lemma:no-mixed-on-C6}.
    Therefore, we conclude that $X$ is a homogeneous set.
\end{proof}

Let us now show that applying a substitution or a $P_7^-$-join to a vertex or an edge satisfying specific conditions does not break being $\Delta(1, 2, 2)$-free.
First, we consider the situation when we substitute a basic tournament, which is a relatively simple case.
Indeed, we prove a stronger statement: that being $\Delta(1, 2, 2)$-free does not change if we substitute any tournament with at least two vertices for a nice vertex, and this is indeed an equivalent condition.

\begin{lemma} \label{Lemma:nice-vertices}
    Let $T$ be a $\Delta(1, 2, 2)$-free tournament, let $v \in V(T)$, and let 
    $S$ be a $\Delta(1,2,2)$-free tournament with at least two vertices.
    Then a tournament obtained from $T$ by substituting $S$ for $v$ is~$\Delta(1, 2, 2)$-free if and only if $v$ is nice.
\end{lemma}
\begin{proof}
    Let $T'$ be a tournament obtained from $T$ by substituting $S$ for $v$.
    We prove the statement by its contrapositive, that is, the tournament $T'$ is not $\Delta(1, 2, 2)$-free if and only if $v$ is not nice.
    First, suppose that $v$ is not a nice vertex.
    Take three distinct vertices $x, y_1, y_2$ in $T$ such that both~$\{v, x, y_1\}$ and~$\{v, x, y_2\}$ induce cyclic triangles.
    Let $Y$ be a homogeneous set in $T'$ that is created by substituting~$S$ for $v$ in $T$, and let $z_1, z_2 \in Y$ be two distinct vertices.
    Then the set $\{x, y_1, y_2, z_1, z_2\}$ induces a copy of $\Delta(1, 2, 2)$ in $T'$. 
    
    To show the converse, suppose that $T'$ contains a copy of $\Delta(1, 2, 2)$.
    Let $Z = \{x, y_1, y_2, z_1, z_2\}$ be the set of five distinct vertices in $T'$ such that $x \Rightarrow \{y_1, y_2\} \Rightarrow \{z_1, z_2\} \Rightarrow x$, and let $C$ be a homogeneous set in $T'$ which is created by substituting $S$ for $v \in V(T)$.
    Since $T$ is $\Delta(1, 2, 2)$-free, we have $\lvert C \cap Z \rvert \geq 2$; since $C$ is a homogeneous set in $T'$ and $T'[C] = S$ is $\Delta(1,2,2)$-free, either~$C \cap Z = \{y_1, y_2\}$ or $C \cap Z = \{z_1, z_2\}$ holds.
    However, then the vertices in the set $(C \setminus Z) \cup \{v\}$ show that $v$ is not nice in $T$.
\end{proof}
% }

Dealing with a $P_7^-$-join requires more effort.
Before starting the proof, we give another description of bridges, which uses the information about the neighbours of their ends.
We omit the proof as it is straightforward from the definition.

\begin{proposition}
    Let $T$ be a tournament and let $uv \in E(T)$.
    Define four sets $X = N^-(u) \cap N^-(v)$, $Y = N^-(u) \cap N^+(v)$, $Z = N^+(u) \cap N^+(v)$, and $W = N^+(u) \cap N^-(v)$.
    Then $uv$ is a bridge if and only if it satisfies the following conditions:

    \begin{enumerate}[label = (B\arabic*)]
        \item The set $W$ is empty; \label{Cond:Bridge1} 
        \item for each $(P,Q) \in \{(X, Y \cup Z), (X \cup Y, Z)\}$, there is no $x \in P$ with at least two distinct in-neighbours in $Q$; and \label{Cond:Bridge2} 
        \item  for each $(P,Q) \in \{(X, Y \cup Z), (X \cup Y, Z)\}$, there is no $x \in Q$ with at least two distinct out-neighbours in $P$. \label{Cond:Bridge3}
    \end{enumerate}
\end{proposition}

In particular, this implies the following.

\begin{corollary} \label{Corollary:bridge-ends}
    Let $T$ be a tournament and let $uv \in E(T)$ be a bridge.
    Then $u$ is a nice vertex in~$T \setminus v$, and vice versa.
\end{corollary}

We also remark the following.
Even if the edge $uv$ does not satisfy~\labelcref{Cond:Bridge1}, it may happen that applying a $P_7^-$-join to $uv$ does not create a copy of $\Delta(1, 2, 2)$.
However, a vertex in the set $W$ enforces that this operation creates a copy of $P_7$, so the obtained tournament can be generated from another~$\Delta(1, 2, 2)$-free tournament by substituting $P_7$.
To avoid this situation, we focus on the situation where the $P_7^-$-join operation creates a homogeneous pair that is not included in a homogeneous set isomorphic to $P_7$.
Given a $\Delta(1, 2, 2)$-free tournament $T$ and a homogeneous pair~$(A_1, A_2)$ such that~$A_1 \cup A_2$ induces a copy of $P_7^-$, we say $(A_1, A_2)$ is \emph{nested} if there is $X \subseteq V(T)$ such that $A_1 \cup A_2 \subseteq X$ and $T[X]$ is isomorphic to $P_7$.

\begin{lemma}\label{Lemma:bridge-edge}
    Let $T$ be a $\Ptwo$-free tournament and let $uv \in E(T)$.
    Let $\tilde{T}$ be a tournament obtained from $T$ by applying a $P_7^-$-join to $uv$, and let $(A_1, A_2)$ be the homogeneous pair of $\tilde{T}$ created by this operation.
    Then $\tilde{T}$ is $\Delta(1, 2, 2)$-free and $(A_1, A_2)$ is not nested if and only if $uv$ is a bridge.
\end{lemma}

\begin{proof}
    Without loss of generality, let $A_1 = \{a_1, a_2, a_3\}$ and $A_2 = \{a_4, a_5, a_6\}$ so that $(a_1, a_2, \ldots, a_6)$ is the canonical ordering of the copy of $P_7^-$.
    Let $A = A_1 \cup A_2$.
    For one direction, suppose that the edge $uv$ is not a bridge.
    If $uv$ violates~\labelcref{Cond:Bridge1}, then the vertices in $A$ together with a vertex in $W$ induce a copy of $P_7$, so $(A_1, A_2)$ is nested.
    Thus, assume that $uv$ satisfies~\labelcref{Cond:Bridge1}.
    We only show that $\tilde{T}$ contains~$\Delta(1, 2, 2)$ when $uv$ violates~\labelcref{Cond:Bridge2} as a similar consideration works when $uv$ violates~\labelcref{Cond:Bridge3}.

    Suppose that $uv$ does not satisfy~\labelcref{Cond:Bridge2}.
    Take $(P, Q) \in \{(X, Y \cup Z), (X \cup Y, Z)\}$ and a vertex $x \in P$ that has two distinct in-neighbours in $Q$.
    Let $y_1$ and $y_2$ be such in-neighbours of $x$ in $Q$.
    If $x \in X$, then $x \Rightarrow \{a_1, a_2\} \Rightarrow \{y_1, y_2\} \Rightarrow x$; if $x \in Y$, then $x \Rightarrow \{a_4, a_5\} \Rightarrow \{y_1, y_2\} \Rightarrow x$.
    This shows that $\tilde{T}$ contains $\Delta(1, 2, 2)$ and we reach a contradiction.

    For the other direction, suppose that the edge $uv$ satisfies~\labelcref{Cond:Bridge1},~\labelcref{Cond:Bridge2},~and~\labelcref{Cond:Bridge3}.
    The homogeneous pair $(A_1, A_2)$ is not nested by~\labelcref{Cond:Bridge1}, so we only need to show that $\tilde{T}$ is $\Delta(1, 2, 2)$-free.
    Suppose for a contradiction that 
    there is $X \subseteq V(\tilde{T})$ such that $\tilde{T}[X]$ is isomorphic to $\Delta(1, 2, 2)$.
    Let~$X = \{x, y_1, y_2, z_1, z_2\}$ so that $x \Rightarrow \{y_1, y_2\} \Rightarrow \{z_1, z_2\} \Rightarrow x$ in $T$.
    Let $t_1 = \lvert X \cap A_1\rvert$ and $t_2 = \lvert X \cap A_2 \rvert$.
    Then~$0 \leq t_1, t_2 \leq 3$ and $t_1 + t_2 \leq 5$.

    We consider the following cases according to the values of $t_1$ and $t_2$.

    \begin{customcase}{1} \label{Case:no-intersection}
        Either $t_1 = 0$ or $t_2 = 0$.
    \end{customcase}

    Without loss of generality, assume that $t_2 = 0$.
    We may assume that $t_1 \geq 2$ as otherwise, it implies that $T$ contains $\Delta(1, 2, 2)$.
    Let $S_1 = \tilde{T}\setminus A_2$.
    Then $X \subseteq V(S_1)$ and $A_1$ is a homogeneous set of $S_1$, and $S_1$ is obtained from $T \setminus u$ by substituting a cyclic triangle for $v$. However, by
    \Cref{Corollary:bridge-ends}, $v$ is a nice vertex in $T \setminus u$, and so by \Cref{Lemma:nice-vertices}, we have that $S_1$ is $\Delta(1, 2, 2)$-free, a contradiction.

    \begin{customcase}{2} \label{Case:whole-intersection}
        Either $t_1 = 3$ or $t_2 = 3$.
    \end{customcase}
    
    Without loss of generality, assume that $t_2 = 3$.
    Since $P_7^-$ is $\Delta(1, 2, 2)$-free, we have $t_1 \leq 1$; by \Cref{Case:no-intersection}, we assume that $t_1 = 1$.
    Let $v$ be the vertex in $X \setminus (A_1 \cup A_2)$.
    Since $A_2$ is a homogeneous set in $T \setminus A_1$, either the out-neighbourhood or the in-neighbourhood of $v$ contains $A_2$, which induces a cyclic triangle.
    However, since every vertex in $\Delta(1, 2, 2)$ has a transitive out-neighbourhood and in-neighbourhood, we reach a contradiction.

    \begin{customcase}{3} \label{Case:both-intersecting-once}
        $t_1 = 1$ and $t_2 = 1$.
    \end{customcase}
        Let $X \cap A_1 = \{u_1\}$ and $X \cap A_2 = \{u_2\}$. 
        Let $T' = \tilde{T} \setminus ((A_1 \cup A_2) \setminus X))$; it follows that $X \subseteq V(T')$. 
        If $u_2u_1 \in E(\tilde{T})$, then $T'$ is isomorphic to $T$, and so $T$ is not $\Delta(1,2,2)$-free, a contradiction. 
        It follows that $u_1u_2 \in E(\tilde{T})$. 
        Since $uv$ is a bridge in $T$, it follows that $u_1u_2$ is not contains in a cyclic triangle in $T'$. 
        Only two edges of $\Delta(1, 2, 2)$ are not in cyclic triangles; so we have~$\{u_1, u_2\} = \{y_1, y_2\}$ or~$\{u_1, u_2\} = \{z_1, z_2\}$. 
        By symmetry, we may assume the former and that $u_i = y_i$ for $i \in \{1, 2\}$. 
        Then the tournament $T[(X \setminus \{u_1, u_2\}) \cup \{v\}]$ is contained in $T \setminus u$ and isomorphic to $\tilde{T}[X \setminus \{y_2\}]$, mapping~$v$ to $u_1 = y_1$ and all other vertices via the identity map. 
        However, the vertex $u_1 = y_1$ is not nice in~$\tilde{T}[X \setminus \{y_2\}]$, and so the vertex $v$ is not nice in $T \setminus u$, contradicting~\Cref{Corollary:bridge-ends}.

    \begin{customcase}{4}
        Either $t_1=2$ or $t_2=2$.
    \end{customcase}
        Without losing generality, assume that $t_2 = 2$.
        By~\Cref{Case:no-intersection} and~\Cref{Case:whole-intersection}, we assume that $1 \leq t_1 \leq 2$.

    \begin{customsubcase}{4.1}
        $t_1 = 2$.
    \end{customsubcase}
        Let $w$ be the vertex in $X \setminus A$.
        Since $P_7^-$ does not contain the $4$-vertex transitive tournament, we have $w \neq x$. Since $\Delta(1, 2, 2)$ is strongly connected, it follows that $A_1 \Rightarrow \{w\} \Rightarrow A_2$. By symmetry, we may assume that $x \in A_1.$ Since $x$ has exactly one in-neighbour in $\{w\} \cup A_2$, it follows that $z_1 \in A_1$ and $z_2 \in A_2$, where $z_1z_2 \in E(\Delta(1,2,2))$. However, since $A_1 \Rightarrow \{w\}$, it follows that $z_1$ has at most one in-neighbour $X \setminus \{x\}$, a contradiction.  

    \begin{customsubcase}{4.2}
        $t_1 = 1$. 
    \end{customsubcase}

    If $A_1 \cap X \Rightarrow A_2 \cap X$, then neither of the edges from $A_1 \cap X$ to $A_2 \cap X$ are contained in a cyclic triangle in $\tilde{T}$. 
    However, the edges of $\Delta(1,2,2)$ not contained in a cyclic triangle form a matching, a contradiction. 
    It follows that there are vertices $u_i \in A_i \cap X$ for $i \in \{1, 2\}$ such that $u_2u_1 \in E(\tilde{T})$. Let~$u_2' \in (A_2 \cap X) \setminus \{u_2\}$. 

    Since $\{u_1, u_2, u_2'\}$ is not a homogeneous set $\tilde{T}[X]$ (because $\Delta(1,2,2)$ has no homogeneous set of size 3), it follows that $X$ contains a vertex $w$ which is mixed on $u_1, u_2, u_2'$; so $\{u_1\} \Rightarrow \{w\} \Rightarrow \{u_2, u_2'\}$. 
    It follows that $\{u_1, w, u_2\}$ induces a cyclic triangle in $\tilde{T}[X]$, and therefore $u_2' \neq x$ (because $\tilde{T}[X] \setminus x$ is transitive). 
    Moreover, if $u_2u_2' \in E(\tilde{T})$, then in $\tilde{T}[X]$, the vertex $u_2'$ has a cyclic triangle among its in-neighbours, a contradiction. 
    We conclude that $\{u_1\} \Rightarrow \{w, u_2'\} \Rightarrow \{u_2\} \Rightarrow \{u_1\}$, and therefore,~$x \in \{u_1, u_2\}$ and $\{w, u_2'\}$ is a homogeneous set in $X$.  
    Let $w' \in X \setminus \{w, u_1, u_2, u_2'\}$. Since $\Delta(1,2,2)$ is strongly connected, and since $w'$ is not mixed on $\{w, u_2'\}$, it follows that $A_1 \rightarrow \{w, w'\} \rightarrow A_2$. 
    This implies that $u_1$ has three out-neighbours in $X$ and $u_2$ has three in-neighbours in $X$, contrary to the fact that one of them is $x$. This concludes Subcase 4.2. 

    We have shown a contradiction in each cases, so we conclude that $\tilde{T}$ is $\Delta(1, 2, 2)$-free.
\end{proof}

In summary, we prove the following theorem in this section.

\begin{theorem} \label{Thm:homogeneous-main-theorem}
    Let $T$ be a $\Delta(1, 2, 2)$-free tournament and let $X \subseteq V(T)$.
    Then the following statements hold.

    \begin{enumerate}
        \item If $T[X]$ is isomorphic to $T_5$ or $P_7$, then $X$ is a homogeneous set.
        \item If $T[X]$ is isomorphic to $P_7^-$, then the degree partition of $T[X]$ is a homogeneous pair of $T$.
        \item Substituting a basic tournament for a nice vertex or applying a $P_7^-$-join to a bridge in $T$ results in a $\Delta(1, 2, 2)$-free tournament. \label{item:operations-preserving-P2free}
    \end{enumerate}
    Moreover, if $T$ contains a basic tournament, then $T$ can be constructed from a smaller $\Delta(1, 2, 2)$-free tournament by substituting a basic tournament for a nice vertex or applying a $P_7^-$-join to a bridge.
\end{theorem}

Therefore, to prove~\Cref{Thm:main-intro}, the remaining thing is to show that every $\Delta(1, 2, 2)$-free, $T_5$-free, and $P_7^-$-free tournament is a paving tournament.
Since $P_7$ contains $P_7^-$, these assumptions also imply that such tournaments are $P_7$-free.

\section{Finding a paving ordering via reshuffling} \label{Section:Finding-Paving-Orderings}

In this section, we prove a lemma that plays a crucial role in finding a paving ordering of a $\Delta(1, 2, 2)$-free tournament.
Roughly speaking, this lemma says that under some assumptions, we can ``reshuffle'' a small number of consecutive vertices in the ordering so that the resulting ordering satisfies the desired conditions.
As an application of this lemma, we show a special case of our main theorem, in particular, when there is a vertex of small in-degree or out-degree.

Let $T$ be a tournament with $n$ vertices and let $\sigma = (v_1, \ldots, v_n)$ be an ordering of $T$.
For $v_i \in V(T)$, a vertex $v_j$ is a \emph{$\sigma$-left-neighbour} (respectively, \emph{$\sigma$-right-neighbour}) of $v_i$ if $j<i$ (respectively, $j>i$) and~$v_i v_j \in E(B_\sigma(T))$.
If the ordering is clear from the context, we only say the vertex $v_j$ is a \emph{left-neighbour} (respectively, \emph{right-neighbour}) of $v_i$.
A vertex $v \in V(T)$ is \emph{paved in $B_\sigma(T)$} if it has at most one $\sigma$-left-neighbour and $\sigma$-right-neighbour.
We also say that $v$ is \emph{paved in $\sigma$} when $v$ is paved in $B_\sigma(T)$.

Given an ordering $\sigma = (v_1, \ldots, v_n)$ of a tournament $T$ and $k$ distinct vertices $v_{i_1}, \ldots, v_{i_k} \in V(T)$, we say an ordering $\sigma' = (v_1', \ldots, v_n')$ of $T$ is obtained from $\sigma$ by \emph{reordering $v_{i_1}, \ldots, v_{i_k}$} if there exists a permutation $\pi : \{i_1, \ldots, i_k\} \to \{i_1, \ldots, i_k\}$ such that $v_j' = v_{\pi(j)}$ for each $j \in \{i_1, \ldots, i_k\}$ and $v_j = v_j'$ if $j \notin \{i_1, \ldots, i_k\}$.
In particular, when $k=2$ and $\pi$ is not the identity function, we say that $\sigma'$ is obtained from $\sigma$ by \emph{swapping $v_{i_1}$ and $v_{i_2}$}.

The following propositions are immediate from the definitions.

\begin{proposition}\label{Proposition:reordering-interval}
    Let $T$ be a tournament with $n$ vertices and $\sigma = (v_1, \ldots, v_n)$ be an ordering of $T$.
    Suppose that there are integers $1 \leq i < j \leq n$ and a vertex $v \in V(T) \setminus \{v_i, \ldots, v_j\}$ such that $v$ is paved in $\sigma$.
    If $\tau$ is an ordering of $T$ obtained from $\sigma$ by reordering $v_i, \ldots, v_j$, then $v$ is also paved in $\tau$.
\end{proposition}

\begin{proposition} \label{Proposition:paving}
    Let $T$ be a paving tournament with $n$ vertices and $\sigma = (v_1, \ldots, v_n)$ be a paving ordering of $T$.
    Suppose that $v_i v_j \in E(B_\sigma(T))$ for $1 \leq i < j \leq n$.
    Then every vertex $v_t$ such that~$i < t < j$ is adjacent from $v_i$ and adjacent to $v_j$ in $T$.
\end{proposition}

In particular,~\Cref{Proposition:reordering-interval} implies the following.

\begin{lemma} \label{Lemma:only-paving}
	Let $T$ be a tournament, and suppose that there is an ordering $\sigma$ of $T$ such that $B_\sigma(T)$ satisfies~\labelcref{Condition:Paving2} but not~\labelcref{Condition:Paving1}.
	Then $T$ is a paving tournament.
\end{lemma}
\begin{proof} 
    Let $n = \lvert V(T) \rvert$.
    Among all orderings of $T$ satisfying~\labelcref{Condition:Paving2}, let $\tau$ be an ordering such that the number of edges in $B_\tau(T)$ is minimized.
    We claim that $\tau$ satisfies~\labelcref{Condition:Paving1}, that is, $\tau$ is a paving ordering of~$T$.
    Suppose to the contrary that there is $i \in [n-1]$ such that $v_iv_{i+1} \in E(B_\tau(T))$.
    Let~$\tau'$ be an ordering of $T$ obtained from $\tau$ by swapping $v_i$ and $v_{i+1}$.
    Then $E(B_{\tau'}(T)) = E(B_\tau(T)) \setminus \{v_i v_{i+1}\}$ and $\tau'$ satisfies \labelcref{Condition:Paving2}, which is a contradiction since $B_{\tau'}(T)$ has fewer edges than $B_\tau(T)$.
    This proves~\Cref{Lemma:only-paving}.
\end{proof}

The following is the key lemma in this section, which states that if there are three consecutive vertices that form a cyclic triangle and they satisfy certain conditions, then we can simply reorder them to make all of them paved.

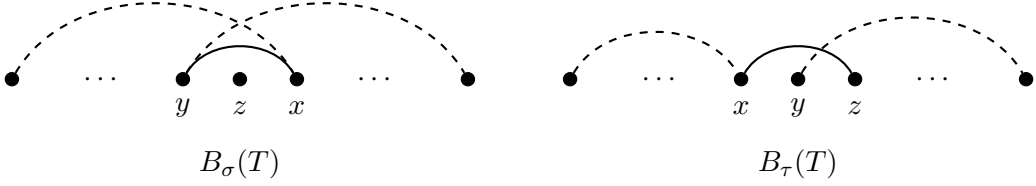
\begin{figure}[t]
    \begin{center}	
    \captionsetup{justification=centering}
        \begin{tikzpicture}[scale=0.75]
        \tikzset{vertex/.style = {shape=circle, fill=black, draw,minimum size=5pt, inner sep=0pt}}
        \tikzset{arc/.style = {->,thick, > = stealth}}
        \tikzset{edge/.style = {thick}}
                         
        \foreach \i in {2,...,4}{
              \node[vertex] (\i) at (\i-3, -1) {};
             }
        \node[vertex] (1) at (-4, -1) {};
        \node[vertex] (5) at (4, -1) {};
        \node at (-1, -1.5) {$y$};
        \node at (0, -1.5) {$z$};
        \node at (1, -1.5) {$x$};

        \node at (-2.4, -1) {$\cdots$};
        \node at (2.4, -1) {$\cdots$};
        \draw[dashed, thick, bend left=60] (1) to (4);
        \draw[edge, bend left=60] (2) to (4);
        \draw[dashed, thick, bend left=60] (2) to (5);

        \node at (0,-2.5) {$B_\sigma(T)$};

    \end{tikzpicture}\quad \quad \quad
    \begin{tikzpicture}[scale=0.75]
        \tikzset{vertex/.style = {shape=circle, fill=black, draw,minimum size=5pt, inner sep=0pt}}
        \tikzset{arc/.style = {->,thick, > = stealth}}
        \tikzset{edge/.style = {thick}}
                         
        \foreach \i in {2,...,4}{
              \node[vertex] (\i) at (\i-3, -1) {};
             }
        \node[vertex] (1) at (-4, -1) {};
        \node[vertex] (5) at (4, -1) {};

        \node at (-1, -1.5) {$x$};
        \node at (0, -1.5) {$y$};
        \node at (1, -1.5) {$z$};
        
        \node at (-2.4, -1) {$\cdots$};
        \node at (2.4, -1) {$\cdots$};
        \draw[dashed, thick, bend left=60] (1) to (2);
        \draw[edge, bend left=60] (2) to (4);
        \draw[dashed, thick, bend left=60] (3) to (5);

        \node at (0,-2.5) {$B_\tau(T)$};
    \end{tikzpicture}
    \caption{Example of reshuffling as in~\Cref{Lemma:Reshuffling}}
    \label{Figure:reshuffling}
    \end{center}
\end{figure} 

\begin{lemma}[Triangle Reshuffling Lemma] \label{Lemma:Reshuffling}
    Let $T$ be a $\Ptwo$-free tournament and let $\sigma$ be an ordering of $T$.
    Suppose that three distinct vertices $a, b, c$ in $T$ satisfy the following.
    \begin{itemize}
        \item The vertices $a, b, c$ induce a cyclic triangle in $T$;
        \item the vertices $a, b, c$ are consecutive in $\sigma$ appearing in that order;
        \item all vertices $a, b, c$ are paved in the ordered graph $B_\sigma(T) - ac$; and
        \item no two distinct vertices in $\{a, b, c\}$ have a common neighbour in $B_\sigma(T)$.
    \end{itemize} 
    Then either $T$ contains $T_5$ or $P_7^-$, or one can find an ordering $\sigma'$, which is obtained from $\sigma$ by reordering~$a, b, c$, so that $a, b, c$ are paved in $B_{\sigma'}(T)$.
\end{lemma}

\begin{proof}
    Let $B = B_\sigma(T)$ and let $X = \{a, b, c\}$.
    The following claim shows that we may assume each of~$a, b, c$ has neighbours not in $X$ in at most one direction. 
    \begin{customclaim}{1} \label{claim:nbrs-one-direction}
        If one of $a, b, c$ has both a left-neighbour and a right-neighbour not in $X$, then $T$ contains~$T_5$.
    \end{customclaim}
    
    \begin{subproof}[Proof of~\Cref{claim:nbrs-one-direction}.]
        Suppose not. 
        After potentially reordering $a, b, c$ and relabeling, we may assume that $b$ has both a left-neighbour $x$ and a right-neighbour $y$ outside $X$. 
        By assumption, neither~$a, b$ nor~$b, c$ have a common neighbour in $B$.
        Thus, we have $b \Rightarrow \{c, x\} \Rightarrow \{a, y\} \Rightarrow b$ if $xy \in E(T)$; otherwise,~$\{a, b, c, x, y\}$ induces $T_5$ in $T$.
    \end{subproof}

    \begin{customclaim}{2} \label{claim:nbrs-same-direction}
        Suppose there are two distinct vertices in $X$ such that one has no left-neighbour outside $X$ and the other has no right-neighbour outside $X$.
        Then there is an ordering $\tau$ of $T$, obtained from $\sigma$ by reordering $a, b, c$, so that all vertices in $X$ are paved in $B_\tau(T)$.
    \end{customclaim}
    
    \begin{subproof}[Proof of~\Cref{claim:nbrs-same-direction}.]
    Let $x \in X$ such that $x$ has no right-neighbour outside $X$, let $z \in X$ such that $x \neq z$ and $z$ has no left-neighbour outside $X$, and let $y \in X$ so that $X = \{x, y, z\}$.

    %   Let $x, y, z\in \{a, b, c\}$ be distinct and such that $x$ has no right-neighbour outside $\{a, b, c\}$ and $z$ has no left-neighbour outside $\{a, b, c\}$, respectively. 
    If $xy, yz, zx \in E(T)$, let $\tau$ be the order obtained from $\sigma$ by reordering the vertices in $X$ so that they appear in the order $x, y, z$.
    Then $y$ is not adjacent to both $x$ and $z$ in $B_\tau(T)$, so $y$ is paved in $\tau$.
    In particular, this implies that all vertices in $X$ are paved in $\tau$, since $x$ has at most one left-neighbour and~$z$ is its only right-neighbour, and $z$ has at most one right-neighbour and $x$ is its only left-neighbour. See~\Cref{Figure:reshuffling} for an illustration. 
    Thus, assume that $xz, yx, zy \in E(T)$.
    By~\Cref{claim:nbrs-one-direction},~$y$ has a neighbour outside $\{x, z\}$ in at most one direction. 
    If $y$ has no neighbour to the left outside~$\{x, z\}$, let $\tau$ be obtained from $\sigma$ by reordering the vertices in $X$ so that they appear in the order $x, z, y$. 
    Otherwise, let $\tau$ be obtained from $\sigma$ by reordering the vertices in $X$ so that they appear in the order $y, x, z$. 
    In either case, $\tau$ is an ordering where all vertices in $X$ are paved.
    \end{subproof}

    By~\Cref{claim:nbrs-same-direction}, we may assume that each of $a, b, c$ has a left-neighbour, say $a',b',c'$, respectively. 
    Now, the five vertices $a, a', b, c, c'$ induce $\Delta(1, 2, 2)$ with $c\Rightarrow \{a, c'\}\Rightarrow \{a', b\}\Rightarrow c$, unless $a'c' \in E(T)$. 
    By applying the same argument for any pair of two vertices in $\{a', b', c'\}$, we have $b'a', c'b' \in E(T)$.
    Therefore, the six vertices $a, a', b, b', c, c'$ induce a copy of $P_7^-$, where $(c',b',a',a,b,c)$ is a canonical ordering of the copy.
\end{proof}

As mentioned earlier, we use this lemma to prove some special cases of the main theorem.
First, we observe that basic tournaments are not paving tournaments.

\begin{proposition} \label{Proposition:C5-or-C6}
    Let $T$ be a tournament containing $T_5$ or $P_7^-$.
    Then for any ordering $\sigma$ of $T$, at least two distinct vertices are not paved in $B_\sigma(T)$.
\end{proposition}
\begin{proof}
    It suffices to prove the statement when $T$ is isomorphic to $T_5$ or $P_7^-$.
    Let $n = \lvert V(T) \rvert$ and fix an ordering $\sigma = (v_1, \ldots, v_n)$ of $T$.
    If $v_1$ is paved in $B_\sigma(T)$, then $d^-(v_1) \leq 1$; similarly, if $v_n$ is paved in $B_\sigma(T)$, then $d^+(v_n) \leq 1$.
    However, every vertex in $T$ has at least two in-neighbours and two out-neighbours, so $v_1$ and $v_n$ are not paved in $B_\sigma(T)$.
\end{proof}

Let $T$ be a tournament and $\sigma$ be an ordering of $T$.
Recall that a vertex is paved in $B_\sigma(T)$ if it has at most one $\sigma$-left-neighbour and at most one $\sigma$-right-neighbour.
We say a vertex in $T$ is \emph{nearly paved in $B_\sigma(T)$} if either it has exactly two $\sigma$-left-neighbours and at most one $\sigma$-right-neighbour, or it has at most one $\sigma$-left-neighbour and exactly two $\sigma$-right-neighbours.

\begin{theorem} \label{Theorem:two-neighbours-same-direction}
	Let $T$ be a $\Delta(1, 2, 2)$-free tournament with $n$ vertices and let $\sigma = (v_1, \ldots, v_n)$ be an ordering of $T$.
    Suppose there is $i \in [n]$ such that all vertices in $T$ except $v_i$ are paved and $v_i$ is nearly paved in $B_\sigma(T)$.
    Then $T$ is a paving tournament.
\end{theorem}

\begin{proof}
Let $G = B_\sigma(T)$.
Without loss of generality, suppose that $v_i$ has two distinct right-neighbours in $G$.
Let $j_1, j_2 \in [n]$ be indices such that $i+1 \leq j_1 < j_2 \leq n$ and $v_i v_{j_1}, v_i v_{j_2} \in E(G)$.
When~$j_1 = i+1$, let $\sigma'$ be an ordering of $T$ obtained from $\sigma$ by swapping $v_i$ and $v_{i+1}$.
By~\Cref{Proposition:reordering-interval}, all vertices in~$V(T) \setminus \{v_i, v_{i+1}\}$ are paved in $\sigma'$.
Moreover, $v_i$ has at most one~$\sigma'$-left-neighbour and exactly one~$\sigma'$-right-neighbour $v_{j_2}$, and $v_{i+1}$ has no $\sigma'$-left-neighbour and at most one $\sigma'$-right-neighbour.
This shows that $v_i, v_{i+1}$ are also paved in $B_\sigma'(T)$, so $T$ is a paving tournament by~\Cref{Lemma:only-paving}.
Hence, we may assume that $j_1 \geq i+2$.

\begin{customclaim}{1} \label{Claim:applying-reshuffling1}
    There is exactly one vertex between $i$ and $j_1$, that is, $j_1 = i+2$.
\end{customclaim}
\begin{subproof}[Proof of~\Cref{Claim:applying-reshuffling1}]
    Suppose to the contrary that $j_1 \geq i+3$.
    Since both $v_{j_1}$ and $v_{j_2}$ are paved in~$G$, $v_{i}$ is the unique $\sigma$-left-neighbour of $v_{j_1}$ and $v_{j_2}$.    
    Thus, $v_{j_1}, v_{j_2}$ are not adjacent to $v_{i+1}, v_{i+2}$ in~$G$ and so $\{v_{i+1}, v_{i+2}\} \Rightarrow \{v_{j_1}, v_{j_2}\}$ in $T$.
    Since $\{v_{j_1}, v_{j_2}\} \Rightarrow v_i \Rightarrow \{v_{i+1}, v_{i+2}\}$, the five vertices~$v_i, v_{i+1}, v_{i+2}, v_{j_1}, v_{j_2}$ induce a copy of $\Delta(1, 2, 2)$ in $T$ and we reach a contradiction.
\end{subproof}

Consider the three vertices $v_i, v_{i+1}, v_{i+2}$.
Observe that $v_{i} v_{i+2} \in E(G)$ but $v_i v_{i+1}, v_{i+1} v_{i+2} \notin E(G)$ by~\Cref{Claim:applying-reshuffling1} and $v_i$ is the only vertex in $T$ which is not paved in $G$.
Furthermore, any two distinct vertices in $\{v_i, v_{i+1}, v_{i+2}\}$ do not have a common neighbour in $G \setminus \{v_i, v_{i+1}, v_{i+2}\}$, since otherwise, the common neighbour is not paved in $G$.
Therefore, by~\Cref{Lemma:Reshuffling}, we conclude that $T$ is a paving tournament.
\end{proof}

\begin{corollary} \label{Corollary:exactly-one-neighbour}
    Let $T$ be a $\Delta(1, 2, 2)$-free tournament.
    Suppose that there exists $x \in V(T)$ such that~$T \setminus x$ is a paving tournament and either $d^+(x) \leq 1$ or $d^-(x) \leq 1$ holds.
    Then $T$ is a paving tournament.
\end{corollary}
\begin{proof}
    Let $n = \lvert V(T) \rvert$.
    It is straightforward that $T$ is a paving tournament if either $n \leq 1$ or $\min\{d^+(x), d^-(x)\} = 0$, so it suffices to prove the case when $n \geq 2$ and $\min\{d^+(x), d^-(x) \} = 1$.
    By possibly reversing all edges, we may assume that $d^-(x) = 1$. 
	Let $\sigma = (v_1, \ldots, v_{n-1})$ be a paving ordering of $T \setminus x$.
    Define an ordering $\sigma' = (x, v_1, \ldots, v_{n-1})$ of $T$ and let $G = B_{\sigma'}(T)$.
    If we let $v_i$ be the in-neighbour of $x$ in $T$ for $i \in [n-1]$, then $v_i$ is the unique neighbour of $x$ in $G$ and so $x$ is paved in $\sigma'$.

    If $v_i$ has no $\sigma'$-left-neighbour other than $x$, then $v_i$ is also paved in $G$.
    Thus, $T$ is a paving tournament by~\Cref{Lemma:only-paving}.
    Otherwise, since $\sigma$ is a paving ordering of $T \setminus x$, there exists a unique index $j \in [i-1]$ such that $v_j$ is a $\sigma'$-left-neighbour of $v_i$.
    This shows that $v_i$ has two $\sigma'$-left-neighbours and at most one $\sigma'$-right-neighbour, and every vertex in $G$ except $v_i$ is paved in $G$.
    Therefore, we conclude that $T$ is a paving tournament by~\Cref{Theorem:two-neighbours-same-direction}.
\end{proof}

\section{Completing the proof of~\Cref{Thm:main}} \label{Section:main-theorem}

In this section, we complete the proof of~\Cref{Thm:main}.
Due to~\Cref{Thm:homogeneous-main-theorem} and~\Cref{Corollary:exactly-one-neighbour}, it suffices to prove the following.

\begin{theorem} \label{Theorem:two-nbrs-in-both-direction}
    Let $T$ be a $\Delta(1, 2, 2)$-free, $T_5$-free, and $P_7^-$-free tournament with at least five vertices.
    Suppose there is a vertex $x \in V(T)$ such that $T \setminus x$ is a paving tournament and $x$ has at least two in-neighbours and at least two out-neighbours.
    Then $T$ is a paving tournament.
\end{theorem}

The proof of~\Cref{Theorem:two-nbrs-in-both-direction} is by induction on the number of vertices in $T$.
More precisely, we guess the correct position of $x$ according to its in-neighbourhood and out-neighbourhood and locally reorder the vertices to obtain a paving ordering of $T$.
This method seems to have too many cases to consider at first glance, but as shown in the following lemma, our assumptions on $T$ and $x$ give a fairly clear idea of where $x$ should go in the paving ordering.

Let $T$ be a tournament with at least five vertices, let $x$ be a vertex of $T$ such that~$d^+(x) \geq 2$ and~$d^-(x) \geq 2$, and let $\sigma = (v_1, \ldots, v_{n-1})$ be an ordering of $T \setminus x$.
For each $i \in \{1, 2\}$, let $\textsf{in}^i(x, \sigma)$ be the $i$th-largest index $j$ such that $v_j$ is an in-neighbour of $x$, and let $\textsf{out}^i(x, \sigma)$ be the $i$th-smallest index $k$ such that $v_k$ is an out-neighbour of $x$.

\begin{lemma}[Fuzzy Lemma] \label{Lemma:Fuzzy-lemma}
    Let $n \geq 5$ be an integer and let $T$ be a $\Delta(1, 2, 2)$-free tournament with~$n$ vertices.
    Suppose there is a vertex $x \in V(T)$ such that $T \setminus x$ is a paving tournament with a paving ordering $\sigma = (v_1, \ldots, v_{n-1})$.
    Then there exists $i \in [n-1]$ such that $x$ has at most one out-neighbour in $\{v_1, \ldots, v_{i-1}\}$ and has at most one in-neighbour in $\{v_{i+4}, \ldots, v_{n-1}\}$.
\end{lemma}
\begin{proof}
	The conclusion immediately follows when $d^+(x) \leq 1$ or $d^-(x) \leq 1$, so we may assume that~$d^+(x) \geq 2$ and~$d^-(x) \geq 2$.
    Let $i = \textsf{out}^2(x, \sigma)$ and $j = \textsf{in}^2(x, \sigma)$.
    We show that $i$ satisfies the desired property.
    Clearly, $x$ has at most one out-neighbour in $\{v_1, \ldots, v_{i-1}\}$.
    In addition, the lemma holds when $i>j$ or $i \geq n-4$, so assume $i<j$ and $i \leq n-5$.
    Let $X = \{v_t : i < t < j\}$.

    \begin{customclaim}{1} \label{Claim:at-most-one-neighbour}
        $\lvert N^+(x) \cap X \rvert \leq 1$ and $\lvert N^-(x) \cap X \rvert \leq 1$.
    \end{customclaim}
    
    \begin{subproof}[Proof of~\Cref{Claim:at-most-one-neighbour}]
        We only show that $\lvert N^+(x) \cap X \rvert \leq 1$ since the proof of the other inequality is similar.
        If $\lvert X \rvert \leq 1$, then the claim trivially holds.
        Assuming $\lvert X \rvert \geq 2$, suppose to the contrary that~$x$ has two distinct out-neighbours $u, u'$ in $X$ (see~\Cref{Figure:at-most-one-neighbour} for an illustration).
        Let $h = \textsf{out}^1(x, \sigma)$ and $k = \textsf{in}^1(x, \sigma)$.
        Since $v_j$ and $v_k$ are paved in $T \setminus x$, the set $L$ of $\sigma$-left-neighbours of $v_j, v_k$ in~$\{u, u' v_h, v_i\}$ consists of at most two vertices.
        However, if we take any two distinct vertices $w, w' \in \{u, u', v_h, v_i\} \setminus L$, then $x \Rightarrow \{w, w'\} \Rightarrow \{v_j, v_k\} \Rightarrow x$.
        This shows that $T$ contains $\Delta(1, 2, 2)$ and we reach a contradiction.
    \end{subproof}

    By~\Cref{Claim:at-most-one-neighbour}, $\lvert X \rvert \leq 2$ and so $j \leq i+3$.
    Therefore, $x$ has at most one in-neighbour in $\{v_{i+4}, \ldots, v_{n-1}\}$, as desired.
\end{proof}

\begin{figure}[t]
    \begin{center}	
        \begin{tikzpicture}[scale=0.75]
            \tikzset{vertex/.style = {shape=circle, fill=black, draw,minimum size=5pt, inner sep=0pt}}
            \tikzset{arc/.style = {->,thick, > = stealth}}
            \tikzset{edge/.style = {thick}}
                             
            \node[vertex, fill=white] (1) at (-5, 0) {};
            \node[vertex] (2) at (-3, 0) {};
            \node[vertex, fill=white] (3) at (-1, 0) {};
            \node[vertex] (4) at (1, 0) {};
            \node[vertex, fill=white] (5) at (3, 0) {};
            \node[vertex, fill=white] (6) at (5, 0) {};
            \node[vertex, fill=white] (v) at (0, -2) {};
            
            \node at (-7, 1) {$B_\sigma(T \setminus x)$};

            \node at (-6, 0) {$\cdots$};
            \node at (-4, 0) {$\cdots$};
            \node at (-2, 0) {$\cdots$};
            \node at (0, 0) {$\cdots$};
            \node at (2, 0) {$\cdots$};
            \node at (4, 0) {$\cdots$};
            \node at (6, 0) {$\cdots$};
            \node at (7, 0) {};

            \node at (-5, 0.5) {$v_h$};
            \node at (-3, 0.5) {$v_i$};
            \node at (-1, 0.5) {$u$};
            \node at (1, 0.5) {$u'$};
            \node at (3, 0.5) {$v_j$};
            \node at (5, 0.5) {$v_k$};
            \node at (0.5, -2.25) {$x$};

            \draw[dotted, thick] (-8, -0.5) -- (8, -0.5);
            
            \draw[arc] (v) to (1);
            \draw[arc] (v) to (2);
            \draw[arc] (v) to (3);
            \draw[arc] (v) to (4);
            \draw[arc] (5) to (v);
            \draw[arc] (6) to (v);

            \draw[edge, bend right=30] (5) to (2);
            \draw[edge, bend right=45] (6) to (4);
        \end{tikzpicture}
    \caption{Proof of~\Cref{Claim:at-most-one-neighbour} when $L = \{u', v_i\}$. The five white vertices induce a copy of $\Delta(1, 2, 2)$ in $T$.}
    \label{Figure:at-most-one-neighbour}
    \end{center}
\end{figure}
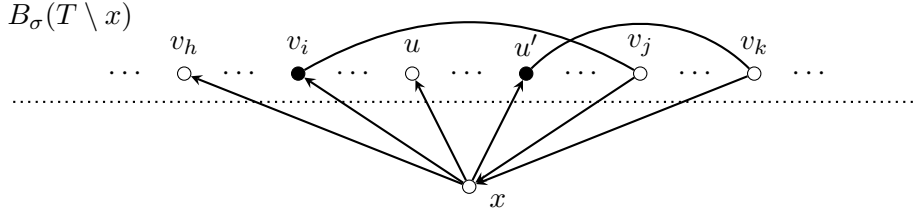 

\Cref{Lemma:Fuzzy-lemma} shows that $x$ ``fits into" the paving ordering of $T \setminus x$ except for at most four vertices whose adjacency to $x$ is unclear (which we call the ``fuzzy part").
Henceforth, we first prove~\Cref{Theorem:two-nbrs-in-both-direction} under the assumption that the fuzzy part does not exist in the following lemma and deal with the remaining case at the end of the section.
Let $T$ be a tournament with $n$ vertices and $\sigma = (v_1, \ldots, v_n)$ be an ordering of $T$.
For two distinct vertices $v_i, v_j$ of $T$, let $\max_{\sigma}\{v_i, v_j\} = v_{\max\{i, j\}}$ and $\min_\sigma \{v_i, v_j\} = v_{\min\{i, j\}}$.

\begin{lemma} \label{Lemma:no-fuzzy-part}
	Let $n \geq 5$ be an integer, let $T$ be a $\{\Delta(1, 2, 2), T_5, P_7^-\}$-free tournament	with $n$ vertices, and let $x$ be a vertex of $T$ such that $d^+(x) \geq 2$ and $d^-(x) \geq 2$.
	Suppose that $T \setminus x$ is a paving tournament with a paving ordering $\sigma = (v_1, \ldots, v_{n-1})$.
	If $\textsf{in}^2(x, \sigma) < \textsf{out}^2(x, \sigma)$, then $T$ is a paving tournament.
\end{lemma}

\begin{proof}
Let $G = B_\sigma(T \setminus x)$, $i = \textsf{out}^2(x, \sigma)$, and $j = \textsf{in}^2(x, \sigma)$.
Suppose that $i>j$.
We consider the following two cases according to the existence of a vertex between $v_i$ and $v_j$ in $\sigma$.

\begin{customcase}{1}
	There is at least one vertex between $v_i$ and $v_j$ in $\sigma$, that is, $i \geq j+2$.
\end{customcase}

Consider the vertex $v_{j+1}$.
If $v_{j+1}$ is an in-neighbour of $x$, then $\textsf{in}^1(x, \sigma) = j+1$.
Thus, $x \Rightarrow \{v_k : k \geq j+2\}$ and $x$ has at most one out-neighbour in $\{v_\ell : \ell \leq j-1\}$.
Consider the ordering 
\[
	\sigma_1 = (v_1, \ldots, v_j, v_{j+1}, x, v_{j+2}, \ldots, v_{n-1})
\]
of $T$ and let $G_1 = B_{\sigma_1}(T)$.
If $x$ has no neighbour in $G_1$, then $\sigma_1$ is a paving ordering of $T$.
Otherwise, let $v_h$ be the unique neighbour of $x$ in $G_1$.
Then $h = \textsf{out}^1(x, \sigma)$ and $v_h$ is the $\sigma_1$-left-neighbour of $x$.
Then $\sigma_1$ is a paving ordering of $T$ unless $v_h$ has a $\sigma_1$-right-neighbour other than $x$.
However, if $\sigma_1$ is not a paving ordering of $T$, then all vertices in $T$ except $v_h$ are paved and $v_h$ is nearly paved in $G_1$.
Thus, by~\Cref{Theorem:two-neighbours-same-direction}, $T$ is a paving tournament.

Suppose that $v_{j+1}$ is an out-neighbour of $x$.
Then $\textsf{out}^1(x, \sigma) = j+1$, so $\{v_k : k \leq j-1\} \Rightarrow x$ and~$x$ has at most one in-neighbour in $\{v_k : k \geq j+2\}$.
In particular, if we let $j' = \textsf{in}^1(x, \sigma)$, then~$i=j+3$ when $j' = j+2$ and $i=j+2$ otherwise.
Now, consider the ordering
\[
    \sigma_2 = (v_1, \ldots, v_j, x, v_{j+1}, \ldots, v_{n-1})
\]
of $T$ and let $G_2 = B_{\sigma_2}(T)$.
Similar to the first case, if $\sigma_2$ is not a paving ordering of $T$, then all vertices other than $v_{j'}$ are paved and $v_{j'}$ is nearly paved in $G_2$.
Thus, we can again apply~\Cref{Theorem:two-neighbours-same-direction} to conclude that $T$ is a paving tournament.
Therefore, we conclude that $T$ is a paving tournament in both cases.

\begin{customcase}{2}
	There is no vertex between $v_i$ and $v_j$ in $\sigma$, that is, $i=j+1$.
\end{customcase}

Consider the ordering $\tau = (v_1, \ldots, v_j, x, v_{j+1}, \ldots, v_n)$ of $T$ and let $H = B_\tau(T)$.
If $\tau$ is a paving ordering of $T$, then we are done.
Thus, assume that this is not true.
Let $h = \textsf{out}^1(x, \sigma)$ and $k = \textsf{in}^1(x, \sigma)$.
Observe that $v_h, v_k$ are the only vertices which are adjacent to $x$ in $H$ and, in particular, $x$ is paved in $H$.
Hence, our assumption shows that $v_{h}$ has a~$\tau$-right-neighbour other than $x$ or $v_{k}$ has a~$\tau$-left-neighbour $x$ other than $x$ in $H$.
Observe that if exactly one of $v_h, v_k$ is not paved in~$H$, then~$T$ is a paving tournament by~\Cref{Theorem:two-neighbours-same-direction}.
Hence, we may assume that both $v_h$ and $v_k$ are not paved in $H$.
Let $v_{h'}$ be a $\tau$-right-neighbour of $v_h$ and $v_{k'}$ be a $\tau$-left-neighbour of $v_k$ in $H$, both of which are different from $x$.

\begin{customclaim}{1} \label{Claim:applying-reshuffling}
    There is a unique vertex between $v_h$ and $\min_\tau\{x, v_{h'}\}$ in $\tau$.
    Similarly, there is a unique vertex between $v_k$ and $\max_\tau\{x, v_{k'}\}$ in $\tau$.
\end{customclaim}
\begin{subproof}[Proof of~\Cref{Claim:applying-reshuffling}]
	We only prove the first statement since the proof of the second statement is similar.
    Let $z = \min_\tau\{x, v_{h'}\}$ and $z'= \max_\tau\{x, v_{h'}\}$.
    If $z = v_{h'}$, then there is a vertex between $v_h$ and $z$ in~$\tau$ by~\labelcref{Condition:Paving1}; if $z=x$, then $v_j$ is between $v_h$ and $z$.
   Now, suppose to the contrary that there are two distinct vertices $u, u'$ between $v_h$ and $z$ in $\tau$.
	Since $\sigma$ is a paving ordering of $T \setminus x$, the property~\labelcref{Condition:Paving2} shows that $v_h \Rightarrow \{u, u'\} \Rightarrow v_{h'}$; by the choice of $h$ and $j$, we have $\{u, u'\} \Rightarrow x$.
	Therefore, the five vertices $u, u', v_h, v_{h'}, x$ induce a copy of $\Delta(1, 2, 2)$ in $T$ since $v_h \Rightarrow \{u, u'\} \Rightarrow \{v_{h'}, x\}, \Rightarrow v_h$ and we reach a contradiction.
\end{subproof}

We consider the following subcases according to the relative positions of $v_{h'}$ and $v_{k'}$ in $\tau$.

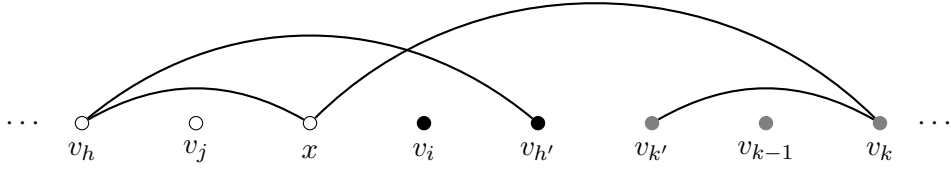
\begin{figure}[t]
    \centering
        \begin{tikzpicture}[scale=0.75]
            \tikzset{vertex/.style = {shape=circle, fill=black, draw,minimum size=5pt, inner sep=0pt}}
            \tikzset{arc/.style = {->,thick, > = stealth}}
            \tikzset{edge/.style = {thick}}

            \node[vertex, fill=white] (h) at (-6, 0) {};
            \node[vertex, fill=white] (i) at (-4, 0) {};
            \node[vertex, fill=white] (x) at (-2, 0) {};
            \node[vertex] (i) at (0, 0) {};
            \node[vertex] (h') at (2, 0) {};
            \node[vertex, color=gray] (k') at (4, 0) {};
            \node[vertex, color=gray] (k-1) at (6, 0) {};
            \node[vertex, color=gray] (k) at (8, 0) {};

            \node at (-6, -0.5) {$v_h$};
            \node at (-4, -0.5) {$v_j$};
            \node at (-2, -0.5) {$x$};
            \node at (0, -0.5) {$v_i$};
            \node at (2, -0.5) {$v_{h'}$};
            \node at (4, -0.5) {$v_{k'}$};
            \node at (6, -0.5) {$v_{k-1}$};
            \node at (8, -0.5) {$v_k$};

            \node at (-7, 0) {$\cdots$};
            \node at (9, 0) {$\cdots$};
            
            \draw[edge, bend left=30] (h) to (x);
            \draw[edge, bend left=40] (h) to (h');
            \draw[edge, bend left=45] (x) to (k);
            \draw[edge, bend left=30] (k') to (k);
        \end{tikzpicture}
    \caption{\Cref{Subcase:2.1} in the proof of~\Cref{Lemma:no-fuzzy-part}. We apply~\Cref{Lemma:Reshuffling} to two sets of three vertices of the same colour.}
    \label{Figure:Subcase-2.1}
\end{figure} 

\begin{customsubcase}{2.1} \label{Subcase:2.1}
    $h' \leq k'$. (See~\Cref{Figure:Subcase-2.1} for an illustration.)
\end{customsubcase}

We can apply~\Cref{Lemma:Reshuffling} twice for vertices $v_h, v_{h+1}, \min_\tau\{x, v_{h'}\}$ and $v_{k-1}, v_k, \max_\tau\{x, v_{k'}\}$ in this case to obtain a paving ordering of $T$.

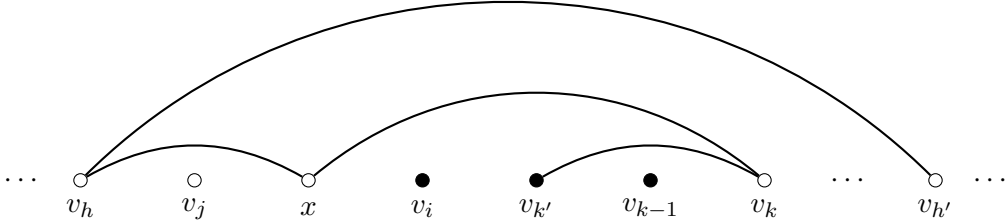
\begin{figure}[t]
    \centering
        \begin{tikzpicture}[scale=0.75]
            \tikzset{vertex/.style = {shape=circle, fill=black, draw,minimum size=5pt, inner sep=0pt}}
            \tikzset{arc/.style = {->,thick, > = stealth}}
            \tikzset{edge/.style = {thick}}

            \node[vertex, fill=white] (h) at (-6, 0) {};
            \node[vertex, fill=white] (i) at (-4, 0) {};
            \node[vertex, fill=white] (x) at (-2, 0) {};
            \node[vertex] (i) at (0, 0) {};
            \node[vertex] (k') at (2, 0) {};
            \node[vertex] (k-1) at (4, 0) {};
            \node[vertex, fill=white] (k) at (6, 0) {};
            \node[vertex, fill=white] (h') at (9, 0) {};

            \node at (-6, -0.5) {$v_h$};
            \node at (-4, -0.5) {$v_j$};
            \node at (-2, -0.5) {$x$};
            \node at (0, -0.5) {$v_{i}$};
            \node at (2, -0.5) {$v_{k'}$};
            \node at (4, -0.5) {$v_{k-1}$};
            \node at (6, -0.5) {$v_{k}$};
            \node at (9, -0.5) {$v_{h'}$};

            \node at (-7, 0) {$\cdots$};
            \node at (10, 0) {$\cdots$};
            \node at (7.5, 0) {$\cdots$};
            
            \draw[edge, bend left=30] (h) to (x);
            \draw[edge, bend left=45] (h) to (h');
            \draw[edge, bend left=40] (x) to (k);
            \draw[edge, bend left=30] (k') to (k);
        \end{tikzpicture}
    \caption{\Cref{Subcase:2.2} in the proof of~\Cref{Lemma:no-fuzzy-part}. Five white vertices induce a copy of $\Delta(1, 2, 2)$ in $T$.}
    \label{Figure:Subcase-2.2}
\end{figure} 

\begin{customsubcase}{2.2} \label{Subcase:2.2}
    Either $h' > k$ or $k' < h$. (See~\Cref{Figure:Subcase-2.2} for an illustration.)
\end{customsubcase}

We only consider the case $h' > k$ as the proof of the other case is similar.
Consider the five vertices~$x, v_h, v_{h'}, v_j, v_k$.
Then $\{x, v_{h'}\} \Rightarrow v_h$ and, since $\sigma$ is a paving ordering, we have $v_h \Rightarrow \{v_j, v_k\} \Rightarrow v_{h'}$.
In addition, by the choice of $j$ and $k$, we have $\{v_j, v_k\} \Rightarrow x$.
In summary, we have $v_h \Rightarrow \{v_j, v_k\} \Rightarrow \{v_{h'}, x\} \Rightarrow v_h$ so $T$ contains a copy of $\Delta(1, 2, 2)$, a contradiction.

\begin{customsubcase}{2.3} \label{Subcase:2.3}
    Either $h < k' < h' \leq j$ or $i \leq k' < h' < k$.
\end{customsubcase}

We only consider the case $h < k' < h' \leq j$ as the proof of the other case is similar.
By~\Cref{Claim:applying-reshuffling}, we have $k' = h+1$, $h' = h+2$, and $i = k-1$.
Thus, we can apply~\Cref{Lemma:Reshuffling} to vertices $v_h, v_{k'}, v_{h'}$ and $v_i, v_k, x$ to obtain a paving ordering of $T$.

\begin{customsubcase}{2.4} \label{Subcase:2.4}
    $h < k' \leq j$ and $i \leq h' < k$.
\end{customsubcase}

\Cref{Claim:applying-reshuffling} implies that $h' = i$ and $k' = j$.
Thus, the five vertices $x, v_h, v_{h'}, v_k, v_{k'}$ induce a copy of $T_5$ in $T$ and we reach a contradiction.

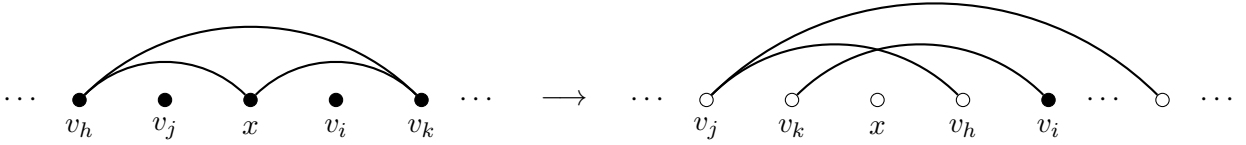
\begin{figure}[t]
    \centering
    \captionsetup{justification=centering}
        \begin{tikzpicture}[scale=0.75]
            \tikzset{vertex/.style = {shape=circle, fill=black, draw,minimum size=5pt, inner sep=0pt}}
            \tikzset{arc/.style = {->,thick, > = stealth}}
            \tikzset{edge/.style = {thick}}

            \node[vertex] (h) at (-8.5, 0) {};
            \node[vertex] (i) at (-7, 0) {};
            \node[vertex] (x) at (-5.5, 0) {};
            \node[vertex] (j) at (-4, 0) {};
            \node[vertex] (k) at (-2.5, 0) {};

            \node[vertex, fill=white] (i1) at (2.5, 0) {};
            \node[vertex, fill=white] (k1) at (4, 0) {};
            \node[vertex, fill=white] (x1) at (5.5, 0) {};
            \node[vertex, fill=white] (h1) at (7, 0) {};
            \node[vertex] (j1) at (8.5, 0) {};
            \node[vertex, fill=white] (w) at (10.5, 0) {};

            \node at (-8.5, -0.5) {$v_h$};
            \node at (-7, -0.5) {$v_j$};
            \node at (-5.5, -0.5) {$x$};
            \node at (-4, -0.5) {$v_i$};
            \node at (-2.5, -0.5) {$v_k$};

            \node at (-9.5, 0) {$\cdots$};
            \node at (-1.5, 0) {$\cdots$};
            \node at (9.5, 0) {$\cdots$};
            \node at (11.5, 0) {$\cdots$};
            \node at (1.5, 0) {$\cdots$};

            \node at (0, 0) {$\longrightarrow$};

            \node at (2.5, -0.5) {$v_j$};
            \node at (4, -0.5) {$v_k$};
            \node at (5.5, -0.5) {$x$};
            \node at (7, -0.5) {$v_h$};
            \node at (8.5, -0.5) {$v_i$};
            
            \draw[edge, bend left=45] (h) to (x);
            \draw[edge, bend left=45] (h) to (k);
            \draw[edge, bend left=45] (x) to (k);

            \draw[edge, bend left=45] (i1) to (h1);
            \draw[edge, bend left=45] (k1) to (j1);
            \draw[edge, bend left=45] (i1) to (w);
        \end{tikzpicture}
    \caption{Ordered graphs $H$ (left) and $H'$ (right) from~\Cref{Subcase:2.5} in the proof of~\Cref{Lemma:no-fuzzy-part}. Five white vertices induce a copy of $\Delta(1, 2, 2)$ in $T$ when $v_j$ has a $\tau'$-right-neighbour in $\{v_t : t \geq k+1\}$.}
    \label{Figure:Subcase-2.5}
\end{figure} 

\begin{customsubcase}{2.5} \label{Subcase:2.5}
    $h' = k$ and $k' = h$. (See~\Cref{Figure:Subcase-2.5} for an illustration.)
\end{customsubcase}

By~\Cref{Claim:applying-reshuffling}, the five vertices $v_h, v_j, x, v_i, v_k$ are consecutive in $\tau$ along this order.
Let $X = \{v_h, v_{i}, x, v_j, v_{k}\}$.
Since $\sigma$ is a paving ordering of $T$, only the adjacency between $v_i$ and $v_j$ is so far undetermined in $H[X]$.
However, if $v_i v_j \in E(H)$, then $T[X]$ is isomorphic to $T_5$, a contradiction.
Hence, $v_i$ and $v_j$ are not adjacent in $H$, that is, $v_j v_i \in E(T)$.

Consider an ordering
\[
    \tau' = (v_1, \ldots, v_{h-1}, v_j, v_k, x, v_h, v_i, v_{k+1}, \ldots, v_{n-1})
\]
of $T$, which is obtained from $\tau$ by reordering the vertices in $X$.
We claim that $\tau'$ is a paving ordering of $T$.
Let $H' = B_{\tau'}(T)$.
By~\Cref{Proposition:reordering-interval}, all vertices in $V(T) \setminus X$ are paved in $H'$.
Moreover, it is not hard to check that $E(H'[X]) = \{v_h v_j, v_i v_k\}$.
Thus, if $\tau'$ is not a paving ordering of $T$, then one of $v_j, v_k$ has a $\tau'$-right-neighbour in $\{v_t : t \geq k+1\}$, or one of $v_h, v_i$ has a $\tau'$-left-neighbour in~$\{v_t : t \leq h-1\}$.
However, 

\begin{itemize}
    \item if $v_j$ has a $\tau'$-right-neighbour $v_\ell$ with $\ell \geq k+1$, then we have $v_j \Rightarrow \{v_k, x\} \Rightarrow \{v_{h}, v_\ell\} \Rightarrow v_j$;
    \item if $v_k$ has a $\tau'$-right-neighbour $v_p$ with $p \geq k+1$, then we have $v_k \Rightarrow \{v_h, x\} \Rightarrow \{v_i, v_p\} \Rightarrow v_k$;
    \item if $v_h$ has a $\tau'$-left-neighbour $v_q$ with $q \leq h-1$, then $v_h \Rightarrow \{v_j, v_q\} \Rightarrow \{v_k, x\} \Rightarrow v_h$; and
    \item if $v_i$ has a $\tau'$-left-neighbour $v_r$ with $r \leq h-1$, then $v_i \Rightarrow \{v_k, v_r\} \Rightarrow \{v_h, x\} \Rightarrow v_i$.
\end{itemize}
Thus, $T$ contains a copy of $\Delta(1, 2, 2)$ in any case and we reach a contradiction.
Hence, we conclude that $\tau'$ is a paving ordering of $T$.

We have shown that $T$ is a paving tournament in all cases and this completes the proof.
\end{proof}

Finally, we complete the proof of~\Cref{Theorem:two-nbrs-in-both-direction}.

\begin{proof}[Proof of~\Cref{Theorem:two-nbrs-in-both-direction}]
	Let $n$ be the number of vertices of $T$, let $\sigma = (v_1, \ldots, v_{n-1})$ be a paving ordering of $T \setminus x$, and let $i = \textsf{out}^2(x, \sigma)$ and $j = \textsf{in}^2(x, \sigma)$. 
	If $i>j$, then $T$ is a paving tournament by~\Cref{Lemma:no-fuzzy-part}.
	Thus, assume $i<j$.
  
	By~\Cref{Lemma:Fuzzy-lemma}, we have $j \leq i+3$.
	Let $h = \textsf{out}^1(x, \sigma)$ and $k = \textsf{in}^1(x, \sigma)$.
	Let $G = B_\sigma(T \setminus x)$.
	Let~$F$ be the set of edges between $\{v_h, v_i\}$ and $\{v_j, v_k\}$ in $G$.
	If $F = \emptyset$, then five vertices $x, v_h, v_i, v_j, v_k$ induce a copy of $\Delta(1, 2, 2)$ in $T$ as $x \Rightarrow \{v_h, v_i\} \Rightarrow \{v_j, v_k\} \Rightarrow x$.
    Thus, $F$ is not empty.
	Moreover, since $\sigma$ is a paving ordering of $T \setminus x$, the edges in $F$ form a matching in $G$.

    We consider the following five cases according to the possible edges in $F$.

	\begin{customcase}{1} \label{Case:nested-matching}
		$F = \{v_h v_k, v_i v_j\}$.
	\end{customcase}
	
	Since $v_i$ and $v_j$ satisfy~\labelcref{Condition:Paving2} in $G \setminus x$, we have $j \geq i+2$.
	Consider the set of five vertices $X_1 = \{v_h, v_i, v_{i+1}, v_k, x\}$.
    Under our assumptions, the adjacency between $x$ and $v_{i+1}$ is as of yet undetermined in $G$.
    However, if $xv_{i+1} \in E(T)$, then $E(G[X_1]) = \{x v_k, v_h v_k\}$, which shows that $v_k \Rightarrow \{v_h, x\} \Rightarrow \{v_i, v_{i+1}\} \Rightarrow v_k$.
    Thus, $T$ contains $\Delta(1, 2, 2)$.
    Otherwise, we have $E(G[X_1]) = \{xv_{i+1}, x v_k, v_h v_k\}$ so~$X_1$ induces a copy of $T_5$ in $T$.
    Thus, we reach a contradiction in either case.
	
	\begin{customcase}{2} \label{Case:crossing-matching}
		$F = \{v_h v_j, v_i v_k\}$.	
	\end{customcase}

	When $j \geq i+2$, let $X_2 = \{x, v_h, v_i, v_{j-1}, v_j\}$.
    By possibly reversing all edges as well as the ordering $\sigma$, we may assume that $xv_{j-1}$ is an edge of $T$.
    Then in $G$, $v_{j-1}$ is not adjacent to any other vertices in~$X_2$ by~\labelcref{Condition:Paving1} and~\labelcref{Condition:Paving2}.
    Thus, we have $v_j \Rightarrow \{v_h, x\} \Rightarrow \{v_i, v_{j-1}\} \Rightarrow v_j$ so $T$ contains $\Delta(1, 2, 2)$.
	Hence, the only possible case is $j=i+1$.
    %Let $X_2 = \{v_h, v_i, v_j, v_k\}$.
	
	\begin{customclaim}{1} \label{Claim:at-most-one-isolated-vertex}
		There is at most one vertex between $v_h$ and $v_i$, and also between $v_j$ and $v_k$.
	\end{customclaim}
	\begin{subproof}[Proof of~\Cref{Claim:at-most-one-isolated-vertex}]
		We only prove this for $v_j$ and $v_k$ as the proof for~$v_h, v_i$ is similar.	
        To show the first statement, suppose to the contrary that $k \geq j+3$.
        Let $Y = \{x, v_i, v_{j+1}, v_{j+2}, v_k\}$.
        Since $\sigma$ is a paving ordering of $T \setminus x$, we have $v_i \Rightarrow \{v_{j+1}, v_{j+2}\} \Rightarrow v_k$ by~\Cref{Proposition:paving}.
        Moreover, by the choices of $j$ and $k$, we have $x \Rightarrow \{v_{j+1}, v_{j+2}\}$; by the choice of $i$ and our assumption, we have $v_k \Rightarrow \{x, v_i\}$.
        In summary, the vertices in $Y$ satisfy $v_k \Rightarrow \{v_i, x \} \Rightarrow \{v_{j+1}, v_{j+2}\} \Rightarrow v_k$ so $Y$ induces a copy of $\Delta(1, 2, 2)$ in $T$ and we reach a contradiction.
        Thus, $k \leq j+2$.
	\end{subproof}

    Observe that by the choice of $h$, $i$, $j$, and $k$, the vertex between $v_h$ and $v_i$, if it exists, is an in-neighbour of $x$ and the vertex between $v_j$ and $v_k$, if it exists, is an out-neighbour of $x$.
	Let~$L = \{x_\ell : \ell \leq h-1\}$ and $R = \{x_m : m \geq k+1\}$.
    Since $\sigma$ is a paving ordering of $T \setminus x$, the vertices~$v_h, v_i$ have no $\tau$-right-neighbour in $R$ and the vertices $v_j, v_k$ have no $\tau$-left-neighbour in $L$.
	However, each~$v_h, v_i$ might have a $\tau$-left-neighbour in $L$ and each $v_j, v_k$ might have a $\tau$-right-neighbour in $R$.

    We consider the following three subcases according to~\Cref{Claim:at-most-one-isolated-vertex}.
	
	\begin{customsubcase}{2.1}
        We have $i=h+1$ and $k = j+1$.
    \end{customsubcase} 

    Let $X_{2, 1} = \{v_h, v_i, v_j, v_k, x\}$, let
	\[
		\tau_{2, 1} = (v_1, \ldots, v_{h-1}, v_j, x, v_h, v_k, v_i, v_{k+1}, \ldots, v_{n-1})
	\]
	be an ordering of $T$, and let $G_{2, 1} $ be the backedge graph of $T$ with respect to $\tau_{2, 1}$.    
    Observe that $E(G_{2, 1}[X_{2, 1}]) = \{v_i v_j, xv_k\}$, and every vertex in $V(T) \setminus X_{2, 1}$ is paved in $G_{2, 1}$ by~\Cref{Proposition:reordering-interval}.
    Thus, if~$\tau_{2, 1}$ is not a paving ordering of $T$, then either $v_i$ has a~$\tau_{2, 1}$-left-neighbour in $L$ or $v_j$ has a~$\tau_{2, 1}$-right-neighbour in $R$.
    However, when the former holds, let $\ell$ be such a left-neighbour of $v_i$ in $L$.
    Then $v_i \Rightarrow \{\ell, v_j\} \Rightarrow \{v_k, x\} \Rightarrow v_i$ so $T$ contains~$\Delta(1, 2, 2)$, which is a contradiction.
    Similarly, when the latter holds, let $r$ be such a right-neighbour of $v_j$ in $R$.
    Then $v_j \Rightarrow \{v_h, x\} \Rightarrow \{r, v_i\} \Rightarrow v_j$ so $T$ contains $\Delta(1, 2, 2)$ and we again reach a contradiction.
    Thus, every vertex in $X_{2, 1}$ is paved in $\tau_{2, 1}$.
	Therefore, by~\Cref{Lemma:only-paving}, $\tau_{2,1}$ is a paving ordering of $T$.
	
    \begin{customsubcase}{2.2}
        We have $i=h+1$ and $k = j+2$, or $i=h+2$ and $k=j+1$.
    \end{customsubcase}

	We only consider the first case as a similar consideration works for the second case.
	By~\Cref{Proposition:paving}, $v_{j+1}$ is adjacent to none of $x, v_h, v_i, v_j, v_k$ in $G$.
	Moreover, $v_{j+1}$ might have a $\tau$-left-neighbour in $L$ or a $\tau$-right-neighbour in $R$.
	Let $X_{2, 2} = \{v_h, v_i, v_j, v_{j+1} v_k, x\}$, let
	\[
		\tau_{2, 2} = (v_1, \ldots, v_{h-1}, v_j, x, v_h, v_k, v_i, v_{j+1}, \ldots, v_{n-1})
	\]
	be an ordering of $T$, and let $G_{2, 2}$ be the backedge graph of $T$ with respect to $\tau_{2, 2}$.
	Then $E(G_{2, 2}[X_{2, 2}]) = \{v_i v_j, x v_k, v_{j+1} v_k\}$.
    Thus, if $\tau_{2, 2}$ is not a paving ordering of $T$, then either one of $v_i, v_{j+1}$ has a $\tau_{2, 2}$-left-neighbour in $L$ or one of $v_j, v_k$ has a~$\tau_{2, 2}$-right-neighbour in $R$.
    However,
    \begin{itemize}
        \item if $v_i$ has a left-neighbour $\ell_1$ in $L$, then we have $v_i \Rightarrow \{v_j, \ell_1\} \Rightarrow \{v_h, v_k\} \Rightarrow v_i$;
        \item if $v_{j+1}$ has a left-neighbour $\ell_2$ in $L$, then $v_{j+1} \Rightarrow \{v_k, \ell_2\} \Rightarrow \{v_i, x\} \Rightarrow v_{j+1}$;
        \item if $v_j$ has a right-neighbour $r_1$ in $R$, then $v_j \Rightarrow \{v_h, v_k\} \Rightarrow \{v_i, r_1\} \Rightarrow v_j$; and
        \item if $v_k$ has a right-neighbour $r_2$ in $R$, then $v_k \Rightarrow \{v_i, x\} \Rightarrow \{v_{j+1}, r_2\} \Rightarrow v_k$.
    \end{itemize}
    Thus, $T$ contains $\Delta(1, 2, 2)$ in every case and we reach a contradiction.
    Hence, we conclude that $\tau_{2, 2}$ is a paving ordering of $T$.
	
    \begin{customsubcase}{2.3}
        We have $i=h+2$ and $k = j+2$.
    \end{customsubcase}

    In this subcase, let $X_{2, 3} = \{x, v_h, v_{h+1}, v_i, v_j, v_{j+1}, v_k\}$.
    Consider the ordering
    \[
		\tau_{2, 3} = (v_1, \ldots, v_{h-1}, v_{h+1}, v_j, x, v_h, v_k, v_i, v_{j+1}, v_{k+1}, \ldots, v_{n-1})
	\]
    of $T$, and let $G_{2, 3}$ be a backedge graph of $T$ with respect to $\tau_{2, 3}$. 
    Note that the adjacency between $v_{h+1}$ and $v_{j+1}$ is not determined yet under the assumption.
    However, if $v_{j+1}v_{h+1} \in E(T)$, then $v_{h+1} \Rightarrow \{v_j, x\} \Rightarrow \{v_h, v_{j+1}\} \Rightarrow v_{h+1}$ so $T$ contains~$\Delta(1, 2, 2)$.
    Thus, $v_{h+1} v_{j+1} \in E(T)$ and $E(G_{2, 3}[X_{2, 3}]) = \{v_h v_{h+1}, v_i v_j, v_{j+1} v_k, v_k x\}$.

    To show that $\tau_{2, 3}$ is a paving ordering of $T$, suppose this is not true.
    Then either one of $v_h, v_i, v_{j+1}$ has a $\tau_{2, 3}$-left-neighbour in $L$ or one of $v_{h+1}, v_j, v_k$ has a $\tau_{2, 3}$-right-neighbour in $R$.
    However,
    \begin{itemize}
        \item if $v_h$ has a left-neighbour $\ell_1$ in $L$, then $v_h \Rightarrow \{\ell_1, v_{h+1}\} \Rightarrow \{v_j, x\} \Rightarrow v_h$;
        \item if $v_i$ has a left-neighbour $\ell_2$ in $L$, then $v_i \Rightarrow \{\ell_2, v_j\} \Rightarrow \{v_h, v_k\} \Rightarrow v_i$;
        \item if $v_{j+1}$ has a left-neighbour $\ell_3$ in $L$, then $v_{j+1} \Rightarrow \{\ell_3, v_k\} \Rightarrow \{v_i, x\} \Rightarrow v_{j+1}$;
        \item if $v_{h+1}$ has a right-neighbour $r_1$ in $R$, then $v_{h+1} \Rightarrow \{v_j, x\} \Rightarrow \{r_1, v_h\} \Rightarrow v_{h+1}$;
        \item if $v_{j}$ has a right-neighbour $r_2$ in $R$, then $v_j \Rightarrow \{v_h, v_k\} \Rightarrow \{r_2, v_i\} \Rightarrow v_j$;
        \item if $v_{k}$ has a right-neighbour $r_3$ in $R$, then $v_k \Rightarrow \{v_i, x\} \Rightarrow \{r_3, v_{j+1}\} \Rightarrow v_k$.
    \end{itemize}
    Thus, $T$ contains $\Delta(1, 2, 2)$ in every case and we reach a contradiction.
    Hence, we conclude that $\tau_{2, 3}$ is a paving ordering of $T$.
	
	\begin{customcase}{3}
		$F = \{v_h v_j\}$ or $F = \{v_i v_k\}$.	
	\end{customcase}

    There is symmetry between the two cases, so we only consider the first case.
    First, we show the following two claims.

    \begin{customclaim}{1} \label{Claim:applying-reshuffling2}
        There is no vertex between $v_i$ and $v_j$, that is, $j=i+1$.
    \end{customclaim}
    \begin{subproof}[Proof of~\Cref{Claim:applying-reshuffling2}]
    
        Suppose to the contrary that $j \geq i+2$. Let $Y' = \{v_h, v_i, v_{i+1}, v_j, x\}$.
        Note that only the adjacency between $x$ and $v_{i+1}$ is undetermined in $T[Y']$ under our assumptions.
	    However, if~$xv_{i+1} \in E(T)$, then $v_j \Rightarrow \{v_h, x\} \Rightarrow \{v_i, v_{i+1}\} \Rightarrow v_j$.
        Thus,~$T[Y']$ is isomorphic to~$\Delta(1, 2, 2)$. Otherwise, $T[Y']$ is isomorphic to~$T_5$.
	    Thus, we reach a contradiction in either case, so we conclude that~$j=i+1$.
    \end{subproof}

    \begin{customclaim}{2} \label{Claim:applying-reshuffling3}
        There is at most one vertex between $v_h$ and $v_i$ in $\tau$, that is, $i \leq h+2$.
    \end{customclaim}
    \begin{subproof}[Proof of~\Cref{Claim:applying-reshuffling3}]
        Suppose to the contrary that $i \geq h+3$. By~\Cref{Proposition:paving} applied to $v_h v_j \in E(G)$, we have $v_h \Rightarrow \{v_{h+1}, v_{h+2}\} \Rightarrow v_j$.
        Moreover, by the choice of $h$ and $i$, we have $\{v_{h+1}, v_{h+2}\} \Rightarrow x$.
        In summary, we have $v_h \Rightarrow \{v_{h+1}, v_{h+2}\} \Rightarrow \{v_j, x\} \Rightarrow v_h$.
        Thus, the five vertices $v_h, v_{h+1}, v_{h+2}, v_j, x$ induce a copy of $\Delta(1, 2, 2)$ and we reach a contradiction.
    \end{subproof}

    Define an ordering of $T$ as follows.
    \begin{itemize}
        \item When $i = h+1$, take
            \[
                \tau_{3, 1} = (v_1, \ldots, v_{h-1}, v_j, x, v_h, v_i, v_{j+1}, \ldots, v_{n-1});
            \]
        \item When $i = h+2$, take
            \[
                \tau_{3, 2} = (v_1, \ldots, v_{h-1}, v_{h+1}, v_j, x, v_h, v_i, v_{j+1}, \ldots, v_{n-1}).
            \]
    \end{itemize}

    We only prove the case $i=h+2$ since a similar consideration works for the other case.
    Assuming~$i=h+2$, let $X_3 = \{v_h, v_{h+1}, v_i, v_j, x\}$ and let $G_3$ be a backedge graph of $T$ with respect to $\tau_{3, 2}$.
    Then~$E(G_3[X_3]) = \{v_h v_{h+1}, v_i v_j\}$ and $xv_k \in E(G_3)$.
    Suppose that $\tau_{3, 2}$ is a not paving ordering of $T$.
    Then either one of $v_h, v_i, v_k$ has a $\tau_{3, 2}$-left-neighbour or one of $v_{h+1}, v_j$ has a~$\tau_{3, 2}$-right-neighbour, both of which are not in $X_3$.
    However,
    \begin{itemize}
        \item if $v_h$ has a left-neighbour $\ell_1$ in $L$, then $v_h \Rightarrow \{\ell_1, v_{h+1}\} \Rightarrow \{v_j, x\} \Rightarrow v_h$;
        \item if $v_i$ has a left-neighbour $\ell_2$ in $L$, then $v_i \Rightarrow \{\ell_2, v_j\} \Rightarrow \{v_h, x\} \Rightarrow v_i$;
        \item if $v_{h+1}$ has a right-neighbour $r_1$, then $v_{h+1} \Rightarrow \{v_j, x\} \Rightarrow \{r_1, v_h\} \Rightarrow v_{h+1}$; and
        \item if $v_j$ has a right-neighbour $r_2$, then $v_j \Rightarrow \{v_h, x\} \Rightarrow \{r_2, v_i\} \Rightarrow v_j$.
    \end{itemize}
    Thus, the only possibility is that $v_k$ has a left-neighbour in $V(T) \setminus X_3$.
    However, in this case, all vertices other than $v_k$ is paved and $v_k$ is nearly paved in $\tau_{3, 2}$, so $T$ is a paving tournament by~\Cref{Theorem:two-neighbours-same-direction}.
	
	\begin{customcase}{4}
		$F = \{v_h v_k\}$.	
	\end{customcase}

  	Since $x$ is adjacent to $v_i$ and adjacent from $v_j$ in $T$, there exists an integer $i'$ with $i \leq i' \leq j-1$ such that $x$ is adjacent to $v_{i'}$ and adjacent from $v_{i'+1}$.
	Thus, if we let $X_4 = \{x, v_h, v_{i'}, v_{i'+1}, v_k\}$, then~$E(G[X_4]) = \{v_h v_k, xv_{i'+1}, xv_k\}$ so $T[X_4]$ is isomorphic to $T_5$, a contradiction.
	
	\begin{customcase}{5}
		$F = \{v_i v_j\}$.	
	\end{customcase}
	
    In this case, we have $j \geq i+2$ by~\labelcref{Condition:Paving2}.
    Recall that $j \leq i+3$.
    Thus, we consider the following two subcases according to the value of $j$.
    
    \begin{customsubcase}{5.1} \label{Case:5.1}
        When $j=i+2$.
    \end{customsubcase}
    
    By possibly reversing all edges, we may assume that $x v_{i+1} \in E(T)$. 
    Note that $v_h v_k \in E(T)$ as~$F = \{v_i v_j\}$.
    Thus, if $v_{i+1}v_k \in E(T)$, then $x \Rightarrow \{v_h, v_{i+1}\} \Rightarrow \{v_j, v_k\} \Rightarrow x$ so $T$ contains $\Delta(1, 2, 2)$.
    This shows that $v_k v_{i+1} \in E(T)$.

    Note that only the adjacency between $v_h$ and $v_{i+1}$ in $T$ is not determined yet by our assumption.
    If $v_{i+1} v_h \in E(T)$, then the five vertices $v_h, v_i, v_{i+1}, v_j, v_k$ induce a copy of $\Delta(1, 2, 2)$ in $T$ as~$v_{i+1} \Rightarrow \{v_h, v_j\} \Rightarrow \{v_i, v_k\} \Rightarrow v_{i+1}$.
    Thus, $v_h v_{i+1} \in E(T)$.
    Moreover, if $k \geq j+3$, then the five vertices~$x, v_j, v_{j+1}, v_{j+2}, v_k$ induce a copy of $\Delta(1, 2, 2)$ in $T$ as $v_k \Rightarrow \{v_{i+1}, x\} \Rightarrow \{v_{j+1}, v_{j+2}\} \Rightarrow v_k$.
    Thus, $k \leq j+2$.

    We define an ordering of $T$ as follows.
    \begin{itemize}
        \item When $k = j+1$, take
        \[
            \tau_{5}^1 = (v_1, \ldots, v_h, v_{h+1}, \ldots, v_{i-1}, v_j, x, v_i, v_k, v_{i+1}, v_{k+1}, \ldots, v_{n-1});
        \]
        \item When $k = j+2$, take
        \[
            \tau_{5}^2 = (v_1, \ldots, v_h, v_{h+1}, \ldots, v_{i-1}, v_j, x, v_i, v_k, v_{i+1}, v_{j+1}, v_{k+1}, \ldots, v_{n-1}).
        \]
    \end{itemize}
   
    We only show that $\tau_5^2$ is a paving ordering of $T$ when $k=j+2$ since a similar consideration works for the other case.
    Assuming $k=j+2$, let $G_5$ be a backedge graph of $T$ with respect to $\tau_5^2$ and let~$X_5 = \{x, v_h, v_i, v_{i+1}, \ldots, v_k\}$.
    Then $E(G_5[X_5]) = \{v_hx, v_{i+1} v_j, v_{j+1} v_k, v_kx\}$.
    Thus, if $\tau_5^2$ is not a paving ordering of $T$, then either one of~$v_{i+1}, v_{j+1}$ has a $\tau_5^2$-left-neighbour in $V(T) \setminus X_5$ or one of~$v_h, v_j, v_k$ has a $\tau_5^2$-right-neighbour in $V(T) \setminus X_5$.
    However,
    \begin{itemize}
        \item if $v_{i+1}$ has a left-neighbour $\ell_1$ other than $v_j$, then $v_{i+1} \Rightarrow \{\ell_1, v_j\} \Rightarrow \{v_k, x\} \Rightarrow v_{i+1}$;
        \item if $v_{j+1}$ has a left-neighbour $\ell_2$ other than $v_k$ and $v_h$, then $v_{j+1} \Rightarrow \{\ell_2, v_k\} \Rightarrow \{v_{i+1}, x\} \Rightarrow v_{j+1}$;
        \item if $v_h$ is a left-neighbour of $v_{j+1}$, then the six vertices $v_h, v_{i+1}, v_j, v_{j+1}, v_k, x$ induces a copy of $P_7^-$ together with a canonical ordering $(v_h, v_j, x, v_k, v_{i+1}, v_{j+1})$;
        \item if $v_j$ has a right-neighbour $r_1$ other than $v_{i+1}$, then $v_j \Rightarrow \{v_k, x\} \Rightarrow \{r_1, v_{i+1}\} \Rightarrow v_j$; and
        \item if $v_k$ has a right-neighbour $r_2$ other than $v_{j+1}$, then $v_k \Rightarrow \{v_{i+1}, x\} \Rightarrow \{r_2, v_{j+1}\} \Rightarrow v_k$.
    \end{itemize}
    Thus, the only possibility is that $v_h$ has a $\tau_5^2$-right-neighbour which is not equal to $x$ or $v_{j+1}$.
    However, then every vertex other than $v_h$ is paved and $v_h$ is nearly paved in $\tau_5^2$.
    Hence, $T$ is a paving tournament by~\Cref{Theorem:two-neighbours-same-direction}.

    \begin{customsubcase}{5.2}
        When $j=i+3$.
    \end{customsubcase}

    Note that the adjacency between $x$ and each $v_{i+1}, v_{i+2}$ in $T$ is not determined yet by our assumptions.
    However, if~$x \Rightarrow \{v_{i+1}, v_{i+2}\}$, then~$v_j \Rightarrow \{v_i, x\} \Rightarrow \{v_{i+1}, v_{i+2}\} \Rightarrow v_j$; if~$\{v_{i+1}, v_{i+2}\} \Rightarrow x$, then~$v_i \Rightarrow \{v_{i+1}, v_{i+2}\} \Rightarrow \{x, v_j\} \Rightarrow v_i$.
    Thus, exactly one of $v_{i+1}, v_{i+2}$ is an in-neighbour of $x$.

    Suppose that $v_{i+2}$ is an in-neighbour of $x$.
    Then $x v_{i+1} \in E(T)$ so the five vertices $x, v_i, v_{i+1}, v_{i+2}, v_j$ induce a copy of $T_5$ in $T$.
    Thus, we have $v_{i+1}x, xv_{i+2} \in E(T)$.
    Now, consider the five vertices~$v_h, v_{i+2}, v_j, v_k, x$.
    If $v_{i+2}v_k \in E(T)$, then $x \Rightarrow \{v_h, v_{i+2}\} \Rightarrow \{v_j, v_k\} \Rightarrow x$ so $T$ contains a copy of $\Delta(1, 2, 2)$.
    This shows that $v_k v_{i+2} \in E(T)$.
    By applying the same argument to the vertices~$v_h, v_i, v_{i+1}, v_k, x$, we also deduce that $v_{i+1} v_h \in E(T)$.

    By applying a similar argument as in~\Cref{Case:5.1}, we have $i \leq h+2$ and $k \leq j+2$.
    Define an ordering of $T$, which is obtained from $\tau$ by reordering $x, v_h, \ldots, v_k$, as follows.
    \begin{itemize}
        \item When $i=h+1$ and $k=j+1$, take 
        \[
            \tau_5^{1, 1} = (v_1, \ldots, v_{h-1}, v_{i+1}, v_h, v_j, x, v_i, v_k, v_{i+2}, v_{k+1}, \ldots, v_{n-1});
        \]
        \item When $i=h+2$ and $k=j+1$, take 
        \[
            \tau_5^{2, 1} = (v_1, \ldots, v_{h-1}, v_{h+1}, v_{i+1}, v_h, v_j, x, v_i, v_k, v_{i+2}, v_{k+1}, \ldots, v_{n-1});
        \]
        \item When $i=h+1$ and $k=j+2$, take 
        \[
            \tau_5^{1, 2} = (v_1, \ldots, v_{h-1}, v_{i+1}, v_h, v_j, x, v_i, v_k, v_{i+2}, v_{j+1}, v_{k+1}, \ldots, v_{n-1});
        \]      
        \item When $i=h+2$ and $k=j+2$, take 
        \[
            \tau_5^{2, 2} = (v_1, \ldots, v_{h-1}, v_{h+1}, v_{i+1}, v_h, v_j, x, v_i, v_k, v_{i+2}, v_{j+1}, v_{k+1}, \ldots, v_{n-1}).
        \]
    \end{itemize}
    We only show that $\tau_5^{2, 2}$ is a paving ordering of $T$ when $i=h+2$ and $k=j+2$ since a similar consideration works for other cases.
    Assuming $i=h+2$ and $k=j+2$, let $G_5'$ be a backedge graph of~$T$ with respect to $\tau_5^{2, 2}$.
    Observe that the adjacencies between $v_{h+1}$ and $v_{i+2}$, $v_{i+1}$ and $v_{j+1}$, and~$v_{h+1}$ and $v_{j+2}$ in $T$ are not determined yet by our assumption.
    However, if $v_{i+2} v_{h+1} \in E(T)$, then the five vertices $v_{h+1}, v_i, v_{i+2}, v_j, v_k$ induce a copy of $\Delta(1, 2, 2)$ in $T$ as~$v_{i+1} \Rightarrow \{v_{h+1}, v_j\} \Rightarrow \{v_i, v_k\} \Rightarrow v_{i+2}$; if $v_{j+1} v_{i+1} \in E(T)$, then the five vertices $v_h, v_i, v_{i+1}, v_j, v_{j+1}$ induce a copy of $\Delta(1, 2, 2)$ as $v_{i+1} \Rightarrow \{v_h, v_{j}\} \Rightarrow \{v_i, v_{j+1}\} \Rightarrow v_{i+1}$.
    Thus, we have $v_{h+1} v_{i+2}, v_{i+1} v_{j+1} \in E(T)$. Furthermore, if~$v_{j+1} v_{h+1} \in E(T)$, then $v_{h+1} \Rightarrow \{x,v_{i+1}\} \Rightarrow \{v_h, v_{j+1}\} \Rightarrow v_{h+1}$ is a copy of $\Delta(1,2,2)$, a contradiction. 

    Let $X_5 = \{v_h, v_{h+1}, v_i, v_{i+1}, v_{i+2}, v_j, v_{j+1}, v_k, x\}$.
    If $\tau_5^{2, 2}$ is not a paving ordering, then either one of $v_h, v_i, v_{i+2}, v_{j+1}$ has a $\tau_5^{2, 2}$-left-neighbour in $V(T) \setminus X_5$, one of $v_{h+1}, v_{i+1}, v_j, v_k$ has a $\tau_5^{2, 2}$-right-neighbour in $V(T) \setminus X_5$. 
    However,
    \begin{itemize}
        \item if $v_{h+1}$ has a right-neighbour $r_1$ other than $v_h$, then $v_{h+1} \Rightarrow \{v_{i+1}, x\} \Rightarrow \{r_1, v_h\} \Rightarrow v_{h+1}$;
        \item if $v_{i+1}$ has a right-neighbour $r_2$ other than $v_i$, then $v_{i+1} \Rightarrow \{v_h, v_j\} \Rightarrow \{r_2, v_i\} \Rightarrow v_{i+1}$;
        \item if $v_j$ has a right-neighbour $r_3$ other than $v_{i+2}$, then $v_j \Rightarrow \{v_i, x\} \Rightarrow \{r_3, v_{i+2}\} \Rightarrow v_j$; and
        \item if $v_k$ has a right-neighbour $r_4$ other than $v_{j+1}$, then $v_k \Rightarrow \{x, v_{i+2}\} \Rightarrow \{r_4, v_{j+1}\} \Rightarrow v_k$.
    \end{itemize}
    Moreover, by symmetry, $T$ contains a copy of $\Delta(1, 2, 2)$ when one of $v_h, v_i, v_{i+2}, v_{j+1}$ has a left-neighbour in $V(T) \setminus X_5$.
    Therefore, we conclude that $\tau_5^{2, 2}$ is a paving ordering of $T$.
     
     We have shown that $T$ is a paving tournament in all cases, and this completes the proof.
\end{proof}

\section{Proof of~\Cref{Thm:main-backedge-intro}} \label{Section:backedge-structure}

In this section, we prove~\Cref{Thm:main-backedge-intro}, which we restate for the reader's convenience.

\begin{customtheorem}{1.1} \label{Thm:Backedge}
    Every $\Delta(1, 2, 2)$-free tournament has a backedge graph in which each component is either a monotone path or isomorphic to one of $H_5, H_6, H_7$.
    In particular, the components isomorphic to $H_5$ or $H_7$ are consecutive, and the vertices in each flock of components isomorphic to $H_6$ are consecutive in the ordering of the backedge graph.
\end{customtheorem}

The first step towards the proof is to observe how the copies of basic tournaments behave when we apply operations.
Roughly speaking, if a $\Delta(1, 2, 2)$-free tournament contains a basic tournament, then we cannot apply operations anymore to the vertices and edges in the copy.
We omit the proof as it is straightforward from the definition of nice vertices and bridges (see \Cref{Subsection:main} for the definitions).

\begin{observation} \label{Observation:No-nested-operations}
    Every vertex in a basic tournament is not nice.
    Moreover, every edge in a basic tournament is not a bridge.
\end{observation}

By combining~\Cref{Observation:No-nested-operations} with~\Cref{Lemma:nice-vertices} and~\Cref{Lemma:bridge-edge}, we deduce the following.

\begin{lemma} \label{Lemma:No-nested-operation}
    Let $T$ be a $\Delta(1, 2, 2)$-free tournament and let $X \subseteq V(T)$.
    If $T[X]$ is isomorphic to a basic tournament, then substituting tournament on at least two vertices for a vertex in $X$ creates a copy of $\Delta(1, 2,2)$.
    Similarly, if $T[X]$ is isomorphic to a basic tournament, then applying a $P_7^-$-join to any edge in $T[X]$ creates a copy of $\Delta(1, 2, 2)$.
\end{lemma}

Say a paving tournament $P$ is an \emph{initial tournament} of a $\Delta(1, 2, 2)$-free tournament $T$ if $T$ is obtained from $P$ by applying a sequence of substitutions of basic tournaments, and $P_7^-$-joins; and~$\lvert V(P)\rvert$ is minimal subject to this (in other words, when a $P_7^-$-join creates a homogeneous set, $P$ only contains one vertex corresponding to this homogeneous set rather than two as would be the case for the~$P_7^-$-join. 
\Cref{Lemma:No-nested-operation} implies that all vertices and edges that we apply operations to originate from an initial tournament.
Given a~$\Delta(1, 2, 2)$-free tournament $T$ and its initial tournament $P$, let $\mathcal{S}(T, P)$ be the set of vertices in $P$ for which we substitute basic tournaments to obtain $T$; let $\mathcal{J}(T, P)$ be the set of edges in $P$ to which we apply a $P_7^-$-join to create a homogeneous pair, the union of which is not a homogeneous set in $T$.
To prove~\Cref{Thm:Backedge}, it suffices to show the following.

\begin{theorem}\label{Thm:initial-isolated}
    Let $T$ be a $\Delta(1, 2, 2)$-free tournament and let $P$ be its initial tournament. Then there is a paving ordering $\sigma$ of $P$ such that every vertex in $\mathcal{S}(T, P)$ is an isolated vertex in $B_\sigma(P)$, and every edge in $\mathcal{J}(T, P)$ is a backedge under $\sigma$ and an isolated edge in $B_\sigma(P)$.
\end{theorem}

To prove~\Cref{Thm:initial-isolated}, we first show that there is an ordering of $T$ such that every vertex in $\mathcal{S}(T, P)$ is isolated in the backedge graph.

\begin{lemma} \label{Lemma:substitution-isolated}
	Let $T$ be a $\Delta(1, 2, 2)$-free tournament and let $P$ be its initial tournament.
	Then there is a paving ordering $\rho$ of $P$ such that every vertex in $\mathcal{S}(T, P)$ is an isolated vertex in $B_\rho(P)$.	
\end{lemma}

\begin{proof}
	Let $n = \lvert V(P) \rvert$ and let $\pi = (v_1, \ldots, v_n)$ be a paving ordering of $P$.
	We show that there is a reordering $\rho$ of $\pi$ such that every vertex in $\mathcal{S}(T, P)$ is an isolated vertex in $B_\rho(P)$ by induction on the number $m$ of vertices in $\mathcal{S}(T, P)$ that are not isolated in $B_\pi(P)$.
	The statement trivially holds when~$m = 0$, so suppose that $m \geq 1$.
	Let $G = B_\pi(P)$ and let $v_i \in \mathcal{S}(T, P)$ be a vertex that is not isolated in $G$.
	Let $C$ be the component of $G$ containing $v_i$.
		
	\begin{claim} \label{Claim:component}
		The component $C$ is a monotone path of length $1$.	
		Moreover, if we let $u$ be the vertex such that $V(C) = \{u, v_i\}$, then there is a unique vertex between $u$ and $v_i$ in $\pi$.
	\end{claim}
	\begin{subproof}[Proof of~\Cref{Claim:component}]
		Suppose to the contrary that $C$ has at least three vertices.
		If $v_i$ has both a left-neighbour $v_h$ and a right-neighbour $v_j$ in $G$, then $h \leq i-2$ and $j \geq i+2$ since $\pi$ is a paving ordering of~$P$.
		However, if we consider the four vertices $v_h, v_{i-1}, v_i, v_j$, then $v_h \Rightarrow \{v_{i-1}, v_j\} \Rightarrow v_i \Rightarrow v_h$ in $T$ so the vertex $v_i$ is not nice, a contradiction. 
		Now, suppose instead that $v_i$ is adjacent to exactly one vertex in~$G$.
		Without loss of generality, assume that $v_i$ has a right-neighbour $v_j$.
		Let $v_k$ be the right-neighbour of $v_j$ in $G$, which exists as $\lvert V(C) \rvert \geq 3$.
		Since $\pi$ is a paving ordering of $P$, we have $j \geq i+2$ and $k \geq j+2$.
		Thus, if we consider the four vertices $v_i, v_{i+1}, v_j, v_k$, then $v_i \Rightarrow \{v_{i+1}, v_k\} \Rightarrow v_j \Rightarrow v_i$ in $P$ so $v_i$ is not nice.
		Hence, we found a contradiction in both cases by~\Cref{Lemma:nice-vertices}, which shows that~$\lvert V(C) \rvert = 2$.
		
		To show the second statement, suppose for a contradiction that there are two distinct vertices $x, y$ between $u$ and $v_i$ in $\pi$.
		Then both $\{u, v_i, x\}$ and $\{u, v_i, y\}$ induce cyclic triangles in $P$, which leads to a contradiction as $v_i$ is a nice vertex.
		This proves the second statement.
	\end{subproof}
	
	Without loss of generality, assume that $V(C) = \{v_i, v_{i+2}\}$.
	Consider an ordering
	\[
		\pi' = (v_1, \ldots, v_{i+2}, v_i, v_{i+1}, \ldots, v_n)
	\]
	of $P$ which is obtained from $\pi$ by reordering $v_i, v_{i+1}, v_{i+2}$.
	Then $\pi'$ is a paving ordering of $P$ unless~$v_{i+1}$ has a $\pi'$-left-neighbour, say $z$.
	However, if this holds, then we have $v_{i+1} \Rightarrow \{v_{i+2}, z\} \Rightarrow v_i \Rightarrow v_{i+1}$ in $P$, which shows that $v_i$ is not a nice vertex.
	Thus, $\pi'$ is a paving ordering of $P$.
	In addition, the vertex $v_i$ is an isolated vertex in $B_{\pi'}(P)$.
	Therefore, the only thing that we need to check is that~$v_{i+1} \notin \mathcal{S}(T, P)$.
	Let $B_1$ and $B_2$ be basic tournaments that are substituted for $v_i$ and $v_{i+1}$, respectively.
	Then $T$ contains~$\Delta(1, B_1, B_2)$, which contains $\Delta(1, 2, 2)$ and we reach a contradiction.
	Thus, $\pi'$ is a paving ordering of $P$ such that the number of non-isolated vertices in $\mathcal{S}(T, P)$ is $m-1$.
	Therefore, by the induction hypothesis, there is a reordering $\rho$ of $\pi$ such that every vertex in $\mathcal{S}(T, P)$ is an isolated vertex in~$B_\rho(T)$.
\end{proof}

    Now, assuming that our ordering satisfies~\Cref{Lemma:substitution-isolated}, we find a reordering that satisfies all desired conditions for $\mathcal{J}(T, P)$ while preserving that every vertex in $\mathcal{S}(T, P)$ is isolated.

\begin{lemma} \label{Lemma:join-backedge}
    Let $T$ be a $\Delta(1, 2, 2)$-free tournament and $P$ be its initial tournament.
    Then there is an ordering $\rho$ of $P$ such that every vertex in $\mathcal{S}(T, P)$ is isolated in $B_\rho(P)$ and every edge in $\mathcal{J}(T, P)$ is a backedge under $\rho$.
\end{lemma}

\begin{proof}
    Let $n = \lvert V(P) \rvert$ and let $\pi = (v_1, \ldots, v_n)$ be an ordering of $P$ such that every vertex in $\mathcal{S}(T, P)$ is isolated in $B_\pi(P)$, which exists by~\Cref{Lemma:substitution-isolated}.
    We prove that there is a reordering of $\pi$ that satisfies all desired conditions by induction on the number $f$ of edges in $\mathcal{J}(T, P)$ that are not backedges under~$\pi$.

    By~\Cref{Lemma:substitution-isolated}, the statement holds when $f = 0$, so assume that $f \geq 1$.
    Take an edge $v_i v_j \in \mathcal{J}(T, P)$ with $i < j$.
    First, consider the case when $j \geq i+2$.
    If there is a vertex $v_k$ such that~$i<k<j$ and~$v_i v_k, v_k v_j \in E(P)$, then $v_k \in N^+(v_i) \cap N^-(v_j)$, which is a contradiction by \Cref{Corollary:homogeneous-C7}.
    Thus, every vertex between $v_i$ and $v_j$ in $\pi$ is adjacent to one of $v_i$ and $v_j$.
    Since $\pi$ is a paving ordering of $P$, this implies that $j = i+3$, $v_i$ is adjacent to $v_{i+2}$ in $B_\pi(P)$, and $v_{i+1}$ is adjacent to $v_{i+3}$ in $B_\pi(P)$.

    If $v_i$ has a $\pi$-left-neighbour, then such a vertex is contained in $N^+(v_i) \cap N^-(v_j)$ and this contradicts~\Cref{Lemma:bridge-edge}.
    Thus, $v_{i+2}$ is the only neighbour of $v_i$ in $B_\pi(P)$; by a similar reasoning, $v_{i+1}$ is the unique neighbour of $v_{i+3}$ in $B_\pi(P)$.
    Consider an ordering
    \[
        \pi' = (v_1, \ldots, v_{i+2}, v_i, v_j, v_{i+1}, \ldots, v_n)
    \]
    of $P$, which is obtained from $\pi$ by reordering $v_i, v_{i+1}, v_{i+2}, v_j$.
    Then $\pi'$ is a paving ordering of $P$ unless $v_{i+1}$ has a $\pi'$-left-neighbour other than $v_{i+2}$ or $v_{i+2}$ has a $\pi'$-right-neighbour other than $v_{i+1}$.
    However, if $v_g v_{i+1} \in E(B_{\pi'}(P))$ with $g \leq i-1$, then $v_{i+1} \Rightarrow \{v_g, v_{i+2}\} \Rightarrow \{v_i, v_j\} \Rightarrow v_{i+1}$ in $P$ so $P$ contains $\Delta(1, 2, 2)$, a contradiction.
    By applying a similar argument, we also deduce a contradiction for the other case  so $\pi'$ is a paving ordering of $P$.
    Moreover, both $v_i$ and $v_j$ are isolated in $B_{\pi'}(P)$.
    This shows that $\{v_i, v_j\}$ is a homogeneous set in $T$, so applying a $P_7^-$-join to the edge $v_i v_j$ results in a homogeneous set that induces a copy of $P_7^-$ in $T$. 
    Now we reach a contradiction, so $j = i+1$.

    Observe that $v_i$ or $v_j$ must have a neighbour in $B_\pi(P)$, as otherwise we obtain a homogeneous set after applying a $P_7^-$-join to $v_i v_j$.
    Moreover, if $v_i$ has a $\pi$-left-neighbour or $v_j$ has a $\pi$-right-neighbour, then such a vertex is contained in $N^+(v_i) \cap N^-(v_j)$ and this contradicts~\Cref{Lemma:bridge-edge}.
    Thus, either $v_i$ has a $\pi$-right-neighbour or $v_j$ has a $\pi$-left-neighbour.
    Without loss of generality, assume that $v_i$ has a~$\pi$-right-neighbour $v_\ell$.

    Observe that $\ell \leq j+2$ since otherwise, we have $v_i \Rightarrow \{v_{j+1}, v_{j+2}\} \Rightarrow v_\ell \Rightarrow v_i$ in $P$ so applying a~$P_7^-$-join to $v_i v_j$ results in a copy of $\Delta(1, 2, 2)$.
    Similarly, if $v_\ell$ has a $\pi$-right-neighbour, then $\ell=j+1$.

    Define an ordering $\pi''$ of $P$ according to the value of $\ell$ as follows:
    \begin{itemize}
        \item When $\ell=j+1$, let $\pi'' = (v_1, \ldots, v_j, v_\ell, v_i, \ldots, v_n)$, which is obtained from $\pi$ by reordering $v_i, v_j, v_\ell$; and
        \item when $\ell=j+2$, let $\pi''=(v_1, \ldots, v_j, v_\ell, v_i, v_{j+1}, \ldots, v_n)$, which is obtained from $\pi$ by reordering $v_i, v_j, v_{j+1}, v_\ell$.
    \end{itemize}

    We only consider the case $\ell=j+2$ as a similar consideration works for the case $\ell=j+1$.
    Observe that $v_i$ is adjacent to $v_j $ and $v_{j+1}$ is adjacent to $v_\ell$ in $B_{\pi''}(P)$.
    If $\pi''$ is not a paving ordering of $P$, then~$v_{j+1}$ has a $\pi''$-left-neighbour other than $v_\ell$, say $v_h$.
    However, then $v_i \Rightarrow v_{j+1} \Rightarrow \{v_h, v_\ell\} \Rightarrow v_i$ in~$P$ so applying a $P_7^-$-join to $v_i v_j$ creates a copy of $\Delta(1, 2, 2)$.
    Thus, the ordering $\pi''$ is a paving ordering of~$P$.

    To apply the induction hypothesis, we show that every vertex in $\mathcal{S}(T, P)$ is isolated in $B_{\pi''}(P)$ and the backedge $v_iv_{\ell}$ we removed in $B_{\pi''}(P)$ is not contained in $\mathcal{J}(T, P)$.
    Observe that if there is a vertex in $\mathcal{S}(T, P)$ that is isolated in $B_\pi(P)$ but not in $B_{\pi''}(P)$, then it is $v_{j+1}$.
    However, as $v_i, v_{j+1}, v_\ell$ induces a cyclic triangle in $P$, substituting a basic tournament for $v_{j+1}$ and applying a~$P_7^-$-join to $v_i v_j$ results in a copy of $\Delta(1, 2, 2)$ in $T$.
    Thus, every vertex in $\mathcal{S}(T, P)$ is isolated in $B_{\pi''}(P)$.
    In addition, the edge~$v_\ell v_i$ is a backedge under $\pi$ but not under $\pi''$.
    However, as $v_i v_j \in \mathcal{J}(T, P)$, the edge $v_\ell v_i$ cannot be an element of $\mathcal{J}(T, P)$.
    
    The arguments above show that $\pi''$ is a paving ordering of $P$ such that every vertex in $\mathcal{S}(T, P)$ is isolated in $B_{\pi''}(P)$ and the number of edges in $\mathcal{J}(T, P)$ that are not backedges under $\pi''$ is~$m-1$.
    Therefore, by the induction hypothesis, there is a reordering $\rho$ of $\pi''$ that satisfies the desired conditions.
\end{proof}

\begin{proof}[Proof of~\Cref{Thm:initial-isolated}]
	Let $n = \lvert V(P) \rvert$ and let $\pi = (v_1, \ldots, v_n)$ be a paving ordering of $P$.
	By~\Cref{Lemma:substitution-isolated} and~\Cref{Lemma:join-backedge}, we may assume that every vertex in $\mathcal{S}(T, P)$ is an isolated vertex in $B_\pi(P)$ and every edge in $\mathcal{J}(T, P)$ is a backedge under $\pi$.
	Let $p$ be the number of edges in $\mathcal{J}(T, P)$ that are not isolated edges in $B_\pi(P)$.
	We prove the statement by induction on $p$.
	
	There is nothing to prove when $p=0$, so assume that $p \geq 1$.
	Take an edge $v_i v_j \in \mathcal{J}(T, P)$ that is not isolated in $B_\pi(P)$.
    Note that $i>j$ by the assumption.
    Without loss of generality, assume that~$v_j$ has a neighbour $v_h$ in $B_\pi(P)$. (If $v_i$ also has a neighbour, we apply an analogous to $v_i$ as well.)
    Then~$h \leq j-1$ by~\Cref{Lemma:bridge-edge}.

    If $j \geq h+3$, then $v_j \Rightarrow v_h \Rightarrow \{v_{h+1}, v_{h+2}\} \Rightarrow v_j$ in $P$ so applying a $P_7^-$-join to $v_i v_j$ results in a copy of $\Delta(1, 2, 2)$.
    Thus, since $\pi$ is a paving ordering of $P$, we have $j = h+2$.
    By a similar reasoning, the vertex $v_h$ does not have a $\pi$-left-neighbour.
    Now, consider an ordering 
    \[
        \pi' = (v_1, \ldots, v_{h+1}, v_j, v_h, \ldots, v_n)
    \]
    of $P$, which is obtained from $\pi$ by reordering $v_h, v_{h+1}, v_j$.
    Then $\pi'$ is a paving ordering of $P$ unless~$v_{h+1}$ has a $\pi'$-right-neighbour $v_\ell$ other than $v_h$.
    If this happens, then $v_{h+1} \Rightarrow v_j \Rightarrow \{v_h, v_\ell\} \Rightarrow v_{h+1}$ in $P$ so applying a $P_7^-$-join to $v_i v_j$ results in a copy of $\Delta(1, 2, 2)$, a contradiction.
    Thus, $\pi'$ is a paving ordering of $P$. 

    To apply the induction hypothesis, we show that the ordering $\pi'$ satisfies our assumptions.
    Clearly, the edge $v_i v_j$ is an isolated edge in $B_{\pi'}(P)$.
    Observe that $v_{h+1}$ is the only vertex that is isolated in~$B_\pi(P)$ but not in $B_{\pi'}(P)$.
    However, since three vertices $v_{h}, v_{h+1}, v_j$ induce a cyclic triangle in~$P$, substituting a basic tournament for $v_{h+1}$ and applying a $P_7^-$-join to $v_i v_j$ results in a copy of~$\Delta(1, 2, 2)$ in~$T$, a contradiction.
    Thus, every vertex in $\mathcal{S}(T, P)$ is an isolated vertex in $B_{\pi'}(P)$.
    In addition, the edge $v_j v_h$~is the only edge that is a backedge under $\pi$ but not in $\pi'$, but $v_j v_h \notin \mathcal{J}(T, P)$ as~$v_i v_j \in \mathcal{J}(T, P)$.
    This shows that every edge in $\mathcal{J}(T, P)$ is a backedge under $\pi'$.
    Finally, observe that $E(B_{\pi'}(P)) \setminus E(B_\pi(P)) =\{v_h v_{h+1}\}$. 
    Each edge in $B_{\pi'}(P)$ incident with $v_h$ or $v_{h+1}$ cannot be contained in $\mathcal{J}(T, P)$, as $v_i v_j \in \mathcal{J}(T, P)$ and the three vertices $v_h, v_{h+1}, v_j$ induce a cyclic triangle in $P$.
    This shows that the ordering $\pi'$ satisfies the desired conditions and the number of edges in $\mathcal{J}(T, P)$ which are not isolated in $B_{\pi'}(P)$ is at most $p-1$.
    Therefore, we complete the proof by applying the induction hypothesis to~$\pi'$.
\end{proof}

\section{Applications: Colouring, Transitive Set, and Triangle Packing} \label{Section:Application}

In this section, we prove~\Cref{Thm:colouring-intro},~\Cref{Thm:transitive-intro}, and~\Cref{Thm:triangle-packing-intro}. 
We first prove~\Cref{Thm:colouring-intro} and~\Cref{Thm:transitive-intro}, which we restate below for the reader's convenience.

\begin{customtheorem}{1.3} \label{Thm:Colouring}
    Every $\Delta(1, 2, 2)$-free tournament is $3$-colourable, and it is $2$-colourable if and only if it is $P_7$-free.
\end{customtheorem}

\begin{customtheorem}{1.4} \label{Thm:Transitive}
    Every $\Delta(1, 2, 2)$-free tournament $T$ on $n$ vertices satisfies $\vec{\alpha}(T) \geq \frac{3n}{7}$.
\end{customtheorem}

The tightness of~\Cref{Thm:Transitive} follows from a tournament obtained from a transitive tournament by substituting $P_7$ for each vertex, which is not hard to verify as $\vec{\alpha}(P_7) = 3$.

To prove these theorems, we investigate the structural properties of the backedge graphs of~$\Delta(1, 2, 2)$-free tournaments.
The idea is as follows.
As the next lemma states, if one can properly colour a backedge graph of a tournament with a small number of colours, then the chromatic number of a tournament is also small.
Here, a \emph{proper $k$-colouring} of an ordered graph is a proper $k$-colouring of its underlying graph.
In this paper, we do not omit the term ``proper" to distinguish it from colouring tournaments.

\begin{lemma}[Folklore] \label{Lemma:Colouring-backedges}
    Let $k \geq 1$ be an integer, let $T$ be a tournament, and let $G$ be a backedge graph of $T$.
    If $G$ admits a proper $k$-colouring,  then $T$ is $k$-colourable.
\end{lemma}
\begin{proof}
    By identifying the vertices of $G$ with their isomorphic image in the backedge graph of $T$, we may assume that $V(G) = V(T)$.
    Let $f : V(G) \to [k]$ be a proper $k$-colouring of $G$.
    For each $i \in [k]$, let~$C_i = \{v \in V(G) : f(v) = i\}$.
    Then each $C_i$ is a stable set in $G$, so it induces a transitive subtournament in $T$.
    Thus, the same function $f : V(T) \to [k]$ gives a $k$-colouring of $T$.
\end{proof}

Thanks to~\Cref{Lemma:Colouring-backedges}, we can analyze the chromatic number of tournaments in terms of their backedge graphs.

\begin{proof}[Proof of~\Cref{Thm:Colouring}]
    Let $T$ be a $\Delta(1, 2, 2)$-free tournament and let $G$ be a backedge graph of $T$ satisfying~\Cref{Thm:main-backedge-intro}.
    Since monotone paths, $H_5$, and $H_6$ admit proper $2$-colourings,~\Cref{Lemma:Colouring-backedges}  implies that $T$ is $2$-colourable when $T$ is $P_7$-free.
    Thus, suppose that $T$ contains $P_7$.
    It is not hard to see that $\chi(P_7) = 3$, so $\chi(T) \geq 3$.
    Define a $3$-colouring of $T$ as follows.
    If a vertex is contained in a component $C$ of $G$ not isomorphic to $H_7$, then use a proper $2$-colouring of $C$; otherwise, use a $3$-colouring of $P_7$.
    Then this gives a $3$-colouring of $T$ by~\Cref{Thm:Colouring}.
\end{proof}

Now, let's prove~\Cref{Thm:Transitive}. 
The main idea is similar: 
If a tournament admits a backedge graph of small chromatic number, we can find a large transitive subtournament either from a colouring of the tournament or a proper colouring of its backedge graph, and we may selectively choose between the two methods to obtain a better bound.

\begin{proof}[Proof of~\Cref{Thm:Transitive}]
    Let $T$ be a $\Delta(1, 2, 2)$-free tournament on $n$ vertices and let $G$ be a backedge graph of $T$ satisfying~\Cref{Thm:main-backedge-intro}.
    If $T$ is $P_7$-free, then $T$ is $2$-colourable by~\Cref{Thm:Colouring} so $\vec{\alpha}(T) \geq n/2$.
    Hence, suppose that $T$ contains $P_7$.
    Let $C_1, \ldots, C_t$ be components in $C$ isomorphic to $H_7$ and let~$X = \bigcup_{i=1}^t V(C_i)$.
    By the assumption, $G \setminus X$ admits a proper $2$-colouring.
    Thus, there is a stable set $I_1 \subseteq V(G) \setminus X$ such that $\lvert I_1 \rvert \geq \frac{1}{2} (n - \lvert X \rvert) = \frac{1}{2} (n-7t)$.
    Moreover, since $\alpha(H_7) = 3$, there is $I_2 \subseteq X$ which is a stable set in $G$ and with $|I_2| = 3t$.
    
    Let $I = I_1 \cup I_2$.
    Then $I$ is a stable set in $G$, so $I$ induces a transitive subtournament of $T$.
    Moreover, since $n \geq 7t$, we have
    \[
        \lvert I \rvert 
        =\lvert I_1 \rvert + \lvert I_2 \rvert 
        = \frac{1}{2} (n-7t) + 3t 
        = \frac{n}{2} - \frac{t}{2} \geq \frac{3n}{7}.
    \]
    Therefore, we conclude that $T$ contains a transitive tournament of size at least $\frac{3n}{7}$.
\end{proof}

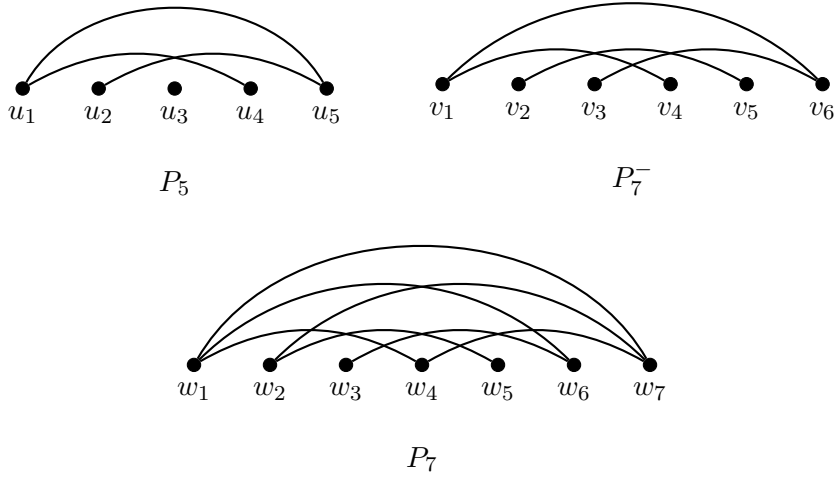
\begin{figure}[t]
    \centering
        \begin{tikzpicture}
            \tikzset{vertex/.style = {shape=circle, fill=black, draw,minimum size=5pt, inner sep=0pt}}
            \tikzset{arc/.style = {->,thick, > = stealth}}
            \tikzset{edge/.style = {thick}}
                             
            \foreach \i in {1,2,...,5}{
              \node[vertex] (\i) at (-3+\i, 0) {};
            }
             
            \foreach \i in {1,2,...,5}{
              \node at (-3+\i, -0.35) {$u_{\i}$};
            }

            \node at (0, -1.25) {$P_5$};

            \draw[edge, bend left=30] (1) to (4);
            \draw[edge, bend left=60] (1) to (5);
            \draw[edge, bend left=30] (2) to (5);
        \end{tikzpicture}  \quad\quad
        \begin{tikzpicture}
            \tikzset{vertex/.style = {shape=circle, fill=black, draw,minimum size=5pt, inner sep=0pt}}
            \tikzset{arc/.style = {->,thick, > = stealth}}
            \tikzset{edge/.style = {thick}}
                             
            \foreach \i in {1,2,...,6}{
              \node[vertex] (\i) at (-3+\i, 0) {};
            }
             
            \foreach \i in {1,2,...,6}{
              \node at (-3+\i, -0.35) {$v_{\i}$};
            }

            \node at (0.5, -1.25) {$P^-_7$};

            \draw[edge, bend left=45] (1) to (6);
            \draw[edge, bend left=30] (2) to (5);
            \draw[edge, bend left=30] (1) to (4);
            \draw[edge, bend left=30] (3) to (6);
        \end{tikzpicture}
        
        \begin{tikzpicture}
            \tikzset{vertex/.style = {shape=circle, fill=black, draw,minimum size=5pt, inner sep=0pt}}
            \tikzset{arc/.style = {->,thick, > = stealth}}
            \tikzset{edge/.style = {thick}}
                             
           \foreach \i in {1,2,...,7}{
              \node[vertex] (\i) at (-4+\i, 0) {};
            }
             
            \foreach \i in {1,2,...,7}{
              \node at (-4+\i, -0.35) {$w_{\i}$};
            }

            \node at (0, -1.25) {$P_7$};

            \draw[edge, bend left=30] (1) to (4);
            \draw[edge, bend left=45] (1) to (6);
            \draw[edge, bend left=60] (1) to (7);
            \draw[edge, bend left=45] (2) to (7);
            \draw[edge, bend left=30] (2) to (5);
            \draw[edge, bend left=30] (3) to (6);
            \draw[edge, bend left=30] (4) to (7);
        \end{tikzpicture} 
    \caption{Backedge graphs of $T_5$, $P_7^-$, and $P_7$ with the minimum number of edges. Note that $(\{v_1, v_2, v_4\}, \{v_3, v_5, v_6\})$ is the degree partition of $P^-_7$.}
    \label{Fig:min-backedge}
\end{figure} 

From now on, we state~\Cref{Thm:triangle-packing-intro} formally and prove it.
As the first step, we first consider the triangle packing in paving tournaments.

\begin{lemma}\label{Lemma:triangle-packing-paving}
    Let $T$ be a paving tournament and let $\sigma$ be a paving ordering of $T$.
    If $B_\sigma(T)$ has $m$ edges, then $\nu(T)\ge\frac{m}{4}$. 
    Furthermore, if $T$ is $\Ptwo$-free, then $\nu(T)\ge\frac{m}{3}$.
\end{lemma}
\begin{proof}
    Suppose not, and let $T$ be a counterexample with the fewest number of vertices. 
    Let $B = B_\sigma(T)$ and let $\sigma = (v_1, \dots, v_n)$. 
    Since $\sigma$ is a paving ordering, $v_1$ has at most one neighbour in $B$. 
    If $v_1$ is an isolated vertex, $T\setminus v_1$ contradicts the minimality of $T$. 
    Thus, $v_1$ has a neighbour, say $v_i$, in $B$, where~$i>2$. 

    Observe that the ordered graph $B \setminus \{v_1, v_2, v_i\}$ is a backedge graph of the tournament $T \setminus \{v_1, v_2, v_i\}$ and it satisfies~\ref{Condition:Paving2}.
    However, it is not guaranteed that $B \setminus \{v_1, v_2, v_i\}$ satisfies~\ref{Condition:Paving1}.
    We consider the following two cases according to this.
    
    \begin{customcase}{1}\label{case:paving}
        $B\setminus \{v_1, v_2, v_i\}$ satisfies~\ref{Condition:Paving1}.
    \end{customcase}
     Note that $v_1$ and $v_2$ have degree at most one in $B$, so $B\setminus \{v_1, v_2, v_i\}$ has at most three fewer edges than $B$. 
     By the minimality of $T$, $T\setminus \{v_1, v_2, v_i\}$ has a triangle packing of size $\frac{m-3}{4}$. 
     But $\{v_1,v_2,v_i\}$ forms a cyclic triangle, so $T$ has a triangle packing of size $1+\frac{m-3}{4} \geq \frac{m}{4}$. 
     
     \begin{customcase}{2}\label{case:not-paving}
         $B\setminus \{v_1, v_2, v_i\}$ does not satisfy~\ref{Condition:Paving1}.
     \end{customcase}
     Then there is an edge $v_jv_k$ in $B$ such that $v_j$ and $v_k$ become consecutive upon deleting $ \{v_1, v_2, v_i\}$. 
     Since there are no vertices in $\sigma$ before $\{v_1,v_2\}$, we conclude that $\{v_j, v_k\} = \{v_{i-1}, v_{i+1}\}$. 
     But then, $B\setminus \{v_1, v_{i-1}, v_i\}$ is a paving backedge graph of $T\setminus \{v_1, v_{i-1}, v_i\}$ with at least $m-4$ edges (losing the edge $v_1v_i$, $v_{i-1}v_{i+1}$ as well as possibly an edge from $v_i$ to its right neighbour and an edge from $v_{i-1}$ to its left neighbour). 
     By minimality, $T\setminus \{v_1, v_{i-1}, v_i\}$ has a triangle packing of size $\frac{m-4}{4}$, so $T$ has a triangle packing of size $\frac{m}{4}$. 

    Now, suppose that $T$ is $\Ptwo$-free.
    We claim that we can always remove a triangle and at most three edges and apply the induction hypothesis. 
    This is immediate in~\Cref{case:paving}. 
    In~\Cref{case:not-paving}, where $v_{i-1}v_{i+1}$ is an edge, note that $v_i$ does not have any right-neighbour $v_l$, as otherwise $v_i\Rightarrow\{v_1, v_{i+1}\}\Rightarrow\{v_{i-1},v_{l}\}\Rightarrow v_i$ forms $\Ptwo$. 
    Thus, we obtain the lower bound $\frac{m}{3}$ on the size of a triangle packing. 
\end{proof}

We also need the following observation.
See~\Cref{Fig:min-backedge} for an illustration.

\begin{observation}\label{Observation:basic-tournament-backedges}
    The minimum number of edges in the backedge graphs of $T_5$, $P_7^-$, and $P_7$ are $3$, $4$, and $7$, respectively.
\end{observation}

Before stating the main theorem, we first explain the ``nice" ordering that we use.
Given a $\Ptwo$-free tournament $T$ and its initial tournament $P$, we say an ordering $\sigma$ of $T$ is \emph{natural} if it is obtained from the ordering of $P$ we obtain from \Cref{Thm:initial-isolated} by performing the appropriate substitutions and joins on $\mathcal{S}(T, P)$ and $\mathcal{J}(T, P)$, using the orderings as shown in Figure \ref{Fig:min-backedge} for substitutions, and the canonical ordering of $P^-_7$ (see~\Cref{Figure:P_7^-}) for $P_7^-$-joins which are not homogeneous sets. 

Finally, we prove our third application result by combining~\Cref{Lemma:triangle-packing-paving} with~\Cref{Thm:initial-isolated}.

\begin{theorem} \label{Thm:triangle-packing}
    Let $T$ be a $\Ptwo$-free tournament and let $\sigma$ be a natural ordering of $T$. 
    If $B_\sigma(T)$ has $m$ edges, then $\nu(T) \geq \frac{2m}{7}$ and this is best possible.
\end{theorem}
\begin{proof}
    Let $P$ be the initial paving tournament for $T$. We proceed by induction on~$|\mathcal{S}(T,P)|+|\mathcal{J}(T,P)|$. 
    The base case~$|\mathcal{S}(T,P)|+|\mathcal{J}(T,P)|=0$ follows from~\Cref{Lemma:triangle-packing-paving}. 
    Assume $|\mathcal{S}(T,P)|+|\mathcal{J}(T,P)|\ge 1$. 
    By the definition of a natural ordering and \Cref{Thm:initial-isolated}, we can choose the last operation so that it is applied at an isolated vertex or isolated edge.
    
    We first handle the case when the last operation is the substitution. 
    Let $T'$ be a tournament so that $T$ be obtained from $T'$ by substituting $T_5$ or $P_7$ for a nice vertex $v$, and let $\sigma'$ be the ordering of $T'$ that is obtained from $\sigma$ in a natural way.
    Let $m'=|E(B_{\sigma'}(T'))|$. 
    It suffices to show that
    
    \begin{equation*}\label{eq:change-ineq}
      \nu(T) - \nu(T')\ge \frac{2}{7}(m-m') \tag{$\dagger$}.
    \end{equation*}
    
    \begin{customcase}{1}
        $T$ is obtained from $T'$ by substituting $T_5$ for $v$.
    \end{customcase}
    By~\Cref{Thm:initial-isolated}, the vertex $v$ is isolated in $B_{\sigma'}(T')$, so all new backedges are internal to the substituted copy. 
    Hence $m=m'+3$ by~\Cref{Observation:basic-tournament-backedges}. 
    Let $X$ be the vertex set of the copy of $T_5$ which we substitute for $v$. 
    For an optimum triangle packing $\mathcal{P}$ of $T'$, we proceed as follows: 
    \begin{itemize}
        \item Add a triangle contained in $X$, say with vertex set $X'$; 
        \item If $v$ is contained in a triangle of $\mathcal{P}$, replace $v$ by a vertex in $X \setminus X'$. 
    \end{itemize}
    This shows that $\nu(T)\ge \nu(T')+1$, as desired.
    
    \begin{customcase}{2}
        $T$ is obtained from $T'$ by substituting $P_7$ for $v$.
    \end{customcase}
    Similarly to the above, $m=m'+7$ by \Cref{Observation:basic-tournament-backedges}, every triangle of $T'$ is a also a triangle of $T$, and the new copy of $P_7$ supplies two disjoint cyclic triangles (plus one more vertex to replace a possible occurrence of $v$ in the triangle packing of $T'$), so again~\Cref{eq:change-ineq} is satisfied.

    The case of performing a $P_7^-$-join follows by a slightly more intricate argument.
    \begin{customcase}{3}
        $T$ is obtained from a smaller tournament by applying the $P_7^-$-join to an edge $uv$.
    \end{customcase}

    Let $X$ be the vertex set of the copy of $P_7^-$ in $T$ that we create. 
    Let $\rho$ be the ordering of $P$ we obtain from~\Cref{Thm:initial-isolated}. 
    Consider the ordering $\rho'$ of $P' = P \setminus \{u, v\}$ obtained from restricting $\rho$. 
    Let $u', u''$ be the vertices immediately before and after $u$ in $\rho$; define $v', v''$ analogously. 
    If $\rho'$ is not a paving ordering of $P'$, then either $u'u''$ or $v'v''$ is an edge of $B_{\rho'}(P')$. 
    For each of them, we consider two cases; we state them here in terms of $u'u''$: 
    \begin{itemize}
        \item If $u''u'$ it not in $\mathcal{J}(T, P)$, we switch their order in $\rho'$; the associated natural ordering of $T \setminus X$ has one fewer backedge than that of $B_\sigma(T) \setminus X$.
        \item If $u''u'$ is in $\mathcal{J}(T, P)$, then we replace them (in $P'$) with a single vertex $\tilde{u}$, which we add to $\mathcal{S}(T \setminus X, P')$; rather than a $P^-_7$-join along $u''u'$, note that substituting $P_7^-$ for $\tilde{u}$ yields the same result. 
        The associated natural backedge graph of $T \setminus X$ has one fewer backedge than $B_\sigma(T) \setminus X$ (because we use the optimum ordering of $P_7^-$ as in~\Cref{Fig:min-backedge} rather than the canonical ordering when $P_7^-$ is a homogeneous set). 
    \end{itemize}
    In each case, we obtain a natural ordering of $T \setminus X$ with at least $m - 7$ backedges (there are 5 backedges in the canonical ordering of $P_7^-$, and for each of $u'u''$ and $v'v''$, we lose at most one further edge). Moreover, $T[X]$ contains two disjoint triangles, so $\nu(T) \geq \nu(T \setminus X) + 2$. 
    Since $|\mathcal{S}(T \setminus X,P')|+|\mathcal{J}(T \setminus X,P')| \leq  |\mathcal{S}(T,P)|+|\mathcal{J}(T,P)| - 1$, the result follows by induction. 

    Finally, the optimality follows from substituting $P_7$ for each vertex of a transitive tournament, since $\nu(P_7)=2$.
\end{proof}

\section*{Acknowledgement}

We are thankful to Xinyue Fan and Cece Henderson for many helpful discussions. 
We would like to thank Sang-il Oum for a careful reading and for helpful comments that improved the presentation of the paper.

\providecommand{\bysame}{\leavevmode\hbox to3em{\hrulefill}\thinspace}
\providecommand{\MR}{\relax\ifhmode\unskip\space\fi MR }
\providecommand{\MRhref}[2]{%
  \href{http://www.ams.org/mathscinet-getitem?mr=#1}{#2}
}

\end{document}